\documentclass{amsart}
\usepackage{amsmath,amsfonts,mathrsfs}

\newif\ifdraft
\draftfalse

\usepackage{amsmath,amsfonts,amsthm,mathrsfs}
\usepackage{amssymb}

\usepackage[unicode,bookmarks]{hyperref}
\usepackage[usenames,dvipsnames]{xcolor}
\hypersetup{colorlinks=true,citecolor=NavyBlue,linkcolor=BrickRed,urlcolor=Orange}

\usepackage[alphabetic,initials]{amsrefs}

\usepackage{enumitem}

\usepackage{chngcntr}

\ifdraft
\usepackage[notcite,notref,color]{showkeys}

\definecolor{labelkey}{gray}{0.5}
\fi

\usepackage{tikz}

\usepackage{tikz-cd}

\usetikzlibrary{matrix,arrows}
\newlength{\myarrowsize} 

\pgfarrowsdeclare{cmto}{cmto}{
	\pgfsetdash{}{0pt} 
	\pgfsetbeveljoin 
	\pgfsetroundcap 
	\setlength{\myarrowsize}{0.6pt}
	\addtolength{\myarrowsize}{.5\pgflinewidth}
	\pgfarrowsleftextend{-4\myarrowsize-.5\pgflinewidth} 
	\pgfarrowsrightextend{.8\pgflinewidth}
}{
	\setlength{\myarrowsize}{0.6pt} 
  	\addtolength{\myarrowsize}{.5\pgflinewidth}  
	\pgfsetlinewidth{0.5\pgflinewidth}
	\pgfsetroundjoin
	\pgfpathmoveto{\pgfpoint{1.5\pgflinewidth}{0}}
	\pgfpatharc{-109}{-170}{4\myarrowsize}
	\pgfpatharc{10}{189}{0.58\pgflinewidth and 0.2\pgflinewidth}
	\pgfpatharc{-170}{-115}{4\myarrowsize+\pgflinewidth}
	\pgfpathclose
	\pgfusepathqfillstroke
	\pgfpathmoveto{\pgfpoint{1.5\pgflinewidth}{0}}
	\pgfpatharc{109}{170}{4\myarrowsize}
	\pgfpatharc{-10}{-189}{0.58\pgflinewidth and 0.2\pgflinewidth}
	\pgfpatharc{170}{115}{4\myarrowsize+\pgflinewidth}
	\pgfpathclose
	\pgfusepathqfillstroke
	\pgfsetlinewidth{2\pgflinewidth}
}

\pgfarrowsdeclare{cmonto}{cmonto}{
	\pgfsetdash{}{0pt} 
	\pgfsetbeveljoin 
	\pgfsetroundcap 
	\setlength{\myarrowsize}{0.6pt}
	\addtolength{\myarrowsize}{.5\pgflinewidth}
	\pgfarrowsleftextend{-4\myarrowsize-.5\pgflinewidth} 
	\pgfarrowsrightextend{.8\pgflinewidth}
}{
	\setlength{\myarrowsize}{0.6pt} 
  	\addtolength{\myarrowsize}{.5\pgflinewidth}  
	\pgfsetlinewidth{0.5\pgflinewidth}
	\pgfsetroundjoin
	\pgfpathmoveto{\pgfpoint{1.5\pgflinewidth}{0}}
	\pgfpatharc{-109}{-170}{4\myarrowsize}
	\pgfpatharc{10}{189}{0.58\pgflinewidth and 0.2\pgflinewidth}
	\pgfpatharc{-170}{-115}{4\myarrowsize+\pgflinewidth}
	\pgfpathclose
	\pgfusepathqfillstroke
	\pgfpathmoveto{\pgfpoint{1.5\pgflinewidth}{0}}
	\pgfpatharc{109}{170}{4\myarrowsize}
	\pgfpatharc{-10}{-189}{0.58\pgflinewidth and 0.2\pgflinewidth}
	\pgfpatharc{170}{115}{4\myarrowsize+\pgflinewidth}
	\pgfpathclose
	\pgfusepathqfillstroke
	\pgfpathmoveto{\pgfpoint{1.5\pgflinewidth-0.3em}{0}}
	\pgfpatharc{-109}{-170}{4\myarrowsize}
	\pgfpatharc{10}{189}{0.58\pgflinewidth and 0.2\pgflinewidth}
	\pgfpatharc{-170}{-115}{4\myarrowsize+\pgflinewidth}
	\pgfpathclose
	\pgfusepathqfillstroke
	\pgfpathmoveto{\pgfpoint{1.5\pgflinewidth-0.3em}{0}}
	\pgfpatharc{109}{170}{4\myarrowsize}
	\pgfpatharc{-10}{-189}{0.58\pgflinewidth and 0.2\pgflinewidth}
	\pgfpatharc{170}{115}{4\myarrowsize+\pgflinewidth}
	\pgfpathclose
	\pgfusepathqfillstroke
	\pgfsetlinewidth{2\pgflinewidth}
}

\pgfarrowsdeclare{cmhook}{cmhook}{
	\pgfsetdash{}{0pt} 
	\pgfsetbeveljoin 
	\pgfsetroundcap 
	\setlength{\myarrowsize}{0.6pt}
	\addtolength{\myarrowsize}{.5\pgflinewidth}
	\pgfarrowsleftextend{-4\myarrowsize-.5\pgflinewidth} 
	\pgfarrowsrightextend{.8\pgflinewidth}
}{
	\setlength{\myarrowsize}{0.6pt} 
  	\addtolength{\myarrowsize}{.5\pgflinewidth}  
 	\pgfsetdash{}{0pt}
	\pgfsetroundcap
	\pgfpathmoveto{\pgfqpoint{0pt}{-4.667\pgflinewidth}}
	\pgfpathcurveto
    {\pgfqpoint{4\pgflinewidth}{-4.667\pgflinewidth}}
    {\pgfqpoint{4\pgflinewidth}{0pt}}
    {\pgfpointorigin}
	\pgfusepathqstroke
}

\newenvironment{diagram*}[2]{%
\[%
\begin{tikzpicture}[>=cmto,baseline=(current bounding box.center),%
	to/.style={->,font=\scriptsize,cap=round},%
	into/.style={cmhook->,font=\scriptsize,cap=round},%
	onto/.style={-cmonto,font=\scriptsize,cap=round},%
	math/.style={matrix of math nodes, row sep=#2, column sep=#1,%
		text height=1.5ex, text depth=0.25ex}]%
}{%
\end{tikzpicture}%
\]%
\ignorespacesafterend%
}

\newcommand{\ZZ}{\mathbb{Z}}

\newcommand{\CC}{\mathbb{C}}

\newcommand{\PP}{\mathbb{P}}

\DeclareMathOperator{\im}{im}

\DeclareMathOperator{\alb}{alb}

\def\overbar#1#2#3{{%
	\setbox0=\hbox{\displaystyle{#1}}%
	\dimen0=\wd0
	\advance\dimen0 by -#2 
	\vbox {\nointerlineskip \moveright #3 \vbox{\hrule height 0.3pt width \dimen0}%
		\nointerlineskip \vskip 1.5pt \box0}%
}}

\makeatletter
\let\@@seccntformat\@seccntformat
\renewcommand*{\@seccntformat}[1]{%
  \expandafter\ifx\csname @seccntformat@#1\endcsname\relax
    \expandafter\@@seccntformat
  \else
    \expandafter
      \csname @seccntformat@#1\expandafter\endcsname
  \fi
    {#1}%
}
\newcommand*{\@seccntformat@subsection}[1]{%
  \textbf{\csname the#1\endcsname.}
}
\makeatother

\makeatletter
\let\@paragraph\paragraph
\renewcommand*{\paragraph}[1]{%
	\vspace{0.3\baselineskip}%
	\@paragraph{\textit{#1}}%
}
\makeatother

\counterwithin{equation}{subsection}
\counterwithout{subsection}{section}
\counterwithin{figure}{subsection}

\newtheorem{theorem}[equation]{Theorem}
\newtheorem*{theorem*}{Theorem}
\newtheorem{lemma}[equation]{Lemma}
\newtheorem*{lemma*}{Lemma}

\newtheorem{proposition}[equation]{Proposition}
\newtheorem*{proposition*}{Proposition}

\theoremstyle{definition}
\newtheorem{definition}[equation]{Definition}
\newtheorem*{definition*}{Definition}
\theoremstyle{remark}
\newtheorem{remark}[equation]{Remark}

\newtheorem*{example*}{Example}
\newtheorem*{problem*}{Problem}

\theoremstyle{plain}

\newcommand{\theoremref}[1]{\hyperref[#1]{Theorem~\ref*{#1}}}
\newcommand{\lemmaref}[1]{\hyperref[#1]{Lemma~\ref*{#1}}}
\newcommand{\definitionref}[1]{\hyperref[#1]{Definition~\ref*{#1}}}
\newcommand{\propositionref}[1]{\hyperref[#1]{Proposition~\ref*{#1}}}
\newcommand{\conjectureref}[1]{\hyperref[#1]{Conjecture~\ref*{#1}}}
\newcommand{\corollaryref}[1]{\hyperref[#1]{Corollary~\ref*{#1}}}
\newcommand{\exampleref}[1]{\hyperref[#1]{Example~\ref*{#1}}}

\makeatletter
\let\old@caption\caption
\renewcommand*{\caption}[1]{%
	\setcounter{figure}{\value{equation}}%
	\stepcounter{equation}%
	\old@caption{#1}\relax%
}
\makeatother

\newcounter{intro}

\newtheorem{intro-conjecture}[intro]{Conjecture}
\newtheorem{intro-corollary}[intro]{Corollary}
\newtheorem{intro-theorem}[intro]{Theorem}

\newcommand{\parref}[1]{\hyperref[#1]{\S\ref*{#1}}}

\makeatletter
\newcommand*\if@single[3]{%
  \setbox0\hbox{${\mathaccent"0362{#1}}^H$}%
  \setbox2\hbox{${\mathaccent"0362{\kern0pt#1}}^H$}%
  \ifdim\ht0=\ht2 #3\else #2\fi
  }
\newcommand*\rel@kern[1]{\kern#1\dimexpr\macc@kerna}
\newcommand*\widebar[1]{\@ifnextchar^{{\wide@bar{#1}{0}}}{\wide@bar{#1}{1}}}
\newcommand*\wide@bar[2]{\if@single{#1}{\wide@bar@{#1}{#2}{1}}{\wide@bar@{#1}{#2}{2}}}
\newcommand*\wide@bar@[3]{%
  \begingroup
  \def\mathaccent##1##2{%
    \if#32 \let\macc@nucleus\first@char \fi
    \setbox\z@\hbox{$\macc@style{\macc@nucleus}_{}$}%
    \setbox\tw@\hbox{$\macc@style{\macc@nucleus}{}_{}$}%
    \dimen@\wd\tw@
    \advance\dimen@-\wd\z@
    \divide\dimen@ 3
    \@tempdima\wd\tw@
    \advance\@tempdima-\scriptspace
    \divide\@tempdima 10
    \advance\dimen@-\@tempdima
    \ifdim\dimen@>\z@ \dimen@0pt\fi
    \rel@kern{0.6}\kern-\dimen@
    \if#31
      \overline{\rel@kern{-0.6}\kern\dimen@\macc@nucleus\rel@kern{0.4}\kern\dimen@}%
      \advance\dimen@0.4\dimexpr\macc@kerna
      \let\final@kern#2%
      \ifdim\dimen@<\z@ \let\final@kern1\fi
      \if\final@kern1 \kern-\dimen@\fi
    \else
      \overline{\rel@kern{-0.6}\kern\dimen@#1}%
    \fi
  }%
  \macc@depth\@ne
  \let\math@bgroup\@empty \let\math@egroup\macc@set@skewchar
  \mathsurround\z@ \frozen@everymath{\mathgroup\macc@group\relax}%
  \macc@set@skewchar\relax
  \let\mathaccentV\macc@nested@a
  \if#31
    \macc@nested@a\relax111{#1}%
  \else
    \def\gobble@till@marker##1\endmarker{}%
    \futurelet\first@char\gobble@till@marker#1\endmarker
    \ifcat\noexpand\first@char A\else
      \def\first@char{}%
    \fi
    \macc@nested@a\relax111{\first@char}%
  \fi
  \endgroup
}
\makeatother

\def\ZZ{{\mathbf Z}}

\def\CC{{\mathbf C}}

\def\PP{{\mathbf P}}
\def\GG{{\mathbf G}}

\newcommand{\I}{\mathcal{I}}

\DeclareMathOperator{\red}{red}
\DeclareMathOperator{\sing}{sing}

\newcommand{\marg}[1]{\normalsize{{\color{red}\footnote{{\color{blue}#1}}}{\marginpar[\vskip -.3cm {\color{red}\hfill\tiny\thefootnote$\rightarrow$}]{\vskip -.3cm{ \color{red}$\leftarrow$\tiny\thefootnote}}}}}
\newcommand{\Yano}[1]{\marg{(Yano) #1}}
\newcommand{\Stefan}[1]{\marg{(Stefan) #1}}
\newcommand{\Mihnea}[1]{\marg{(Mihnea) #1}}

\newif\ifcomments
\commentstrue 	
\commentsfalse 	

\ifcomments
\renewcommand{\Yano}[1]{}
\renewcommand{\Stefan}[1]{}
\renewcommand{\Mihnea}[1]{}
\fi

\begin{document}

\vspace{\baselineskip}


\title[Generic vanishing conjecture in dimension five]{Generic vanishing and minimal cohomology classes on abelian fivefolds}

\author[Casalaina-Martin]{Sebastian Casalaina-Martin }
\address{University of Colorado, Department of Mathematics,  Campus Box 395,
Boulder, CO 80309-0395, USA}
\email{casa@math.colorado.edu}

\author[Popa]{Mihnea Popa}
\address{Department of Mathematics, Northwestern University,
2033 Sheridan Road, Evanston, IL 60208, USA} 
\email{\tt mpopa@math.northwestern.edu}

\author[Schreieder]{Stefan Schreieder}
\address{Mathematical Institute, University of Bonn, Endenicher Allee 60, 53115 Bonn, Germany}
\email{\tt schreied@math.uni-bonn.de}


\thanks{SC was partially supported by Simons Foundation Collaboration Grant for Mathematicians (317572).  
MP was partially supported by the NSF grant DMS-1405516 and by a Simons Fellowship. 
}


\setlength{\parskip}{0.2\baselineskip}

\maketitle

\begin{abstract}
We classify $GV$-subschemes of five-dimensional  ppavs, 
 proving the main conjecture in \cite{PP3} in this case.
This result is implied by a more general statement about subvarieties of minimal cohomology class whose sum is a theta
divisor. 
\end{abstract}

\section{Introduction}

The aim of this paper is to give a proof of the main conjecture in \cite{PP3} in the case of five-dimensional principally polarized complex  abelian varieties (ppavs).   

\begin{intro-theorem}\label{conjecture}
Let $(A, \Theta)$ be an indecomposable ppav of dimension $5$. If $1 \le d \le 3$, then $X$ is a geometrically nondegenerate closed $GV$-subscheme of $A$ of dimension $d$ if and only if one of the following holds: 

\begin{enumerate}
\item[(a)] There is a smooth genus $5$ curve $C$ and an isomorphism $(A,\Theta)\cong (JC,\Theta_C)$ that  identifies $X$ with $W_d(C)$. 
\item[(b)] $d=2$ and there is a smooth cubic threefold $Y$ and an isomorphism $(A,\Theta)\cong (JY,\Theta_Y)$ that identifies $X$ with $F$, the Fano surface of lines on $Y$. 
\end{enumerate}
\end{intro-theorem}

Recall that $X \subset A$ is called a $GV$-subscheme if $\I_X (\Theta)$ is a $GV$-sheaf on $A$ in the sense of \cite{PP4}; concretely, 
this means that 
$${\rm codim}~ \{ \alpha \in {\rm Pic}^0 (A) ~|~ H^i (A, \I_X (\Theta) \otimes \alpha ) \neq 0 \} \ge i \,\,\,\,\,\,{\rm for ~all} \,\,\,\,i \ge 0.$$
This property is a formal analogue of $2$-regularity in projective space, as we will recall  below.    
When $\dim X = 1$ or $3$,  the result in \theoremref{conjecture} already follows from \cite[Theorem C]{PP3}. The (much more difficult) case addressed here is that of $\dim X = 2$; this turns out to be a consequence of the following more general statement, which is the main result of the paper.

\begin{intro-theorem}\label{main}
Let $(A, \Theta)$ be an indecomposable ppav of dimension $5$, and let $V, W\subset A$ be two-dimensional subvarieties, such that $[V] = [W] = \frac{\theta^3}{3!}$ and $V  + W = \Theta$. Then one of the following holds: 

\begin{enumerate}\item[(a)] There is a smooth genus $5$ curve $C$ and an isomorphism $(A,\Theta)\cong (JC,\Theta_C)$ that identifies $V$ with $W_2(C)$.
\item[(b)] There is a smooth cubic threefold $Y$ and an isomorphism $(A,\Theta)\cong (JY,\Theta_Y)$ that identifies $V$ with $F$, the 
Fano surface of lines on $Y$. 
\end{enumerate}
\end{intro-theorem}

In the statements of both theorems we are referring to the Abel--Jacobi embedding of the Fano surface of lines in $JY$, 
which also happens to be isomorphic to ${\rm Alb}(F)$; see \cite{CG}.

\subsection*{Context and previous work}
Recall the combination of the main conjectures in \cite{Debarre} and \cite{PP3}. 

\begin{intro-conjecture}\label{full_conjecture}
Let $(A, \Theta)$ be an indecomposable ppav of dimension $g$, and let  $X$ be a closed subscheme of $A$ of dimension $1 \le d \le g-2$. The following are equivalent:

\begin{enumerate}
\item $X$ is reduced of pure dimension and has minimal cohomology class, i.e. $[X] = \frac{\theta^{g-d}}{(g-d)!}$.
\item $X$ is a geometrically nondegenerate  $GV$-subscheme, i.e. $X$ is geometrically nondegenerate and  $\I_X (\Theta)$ is a $GV$-sheaf on $A$.
\item One of the following holds: 
\begin{enumerate}\item[(a)] There is a smooth genus $g$ curve $C$ and an isomorphism $(A,\Theta)\cong (JC,\Theta_C)$ that identifies $X$ with $W_d(C)$.
\item[(b)] $g=5$, $d =2$, and there is a smooth cubic threefold $Y$ and an isomorphism $(A,\Theta)\cong (JY,\Theta_Y)$ that  identifies $X$ with $F$, the 
Fano surface of lines on $Y$. 
\end{enumerate}
\end{enumerate}
\end{intro-conjecture}

Establishing the equivalence between $(1)$ and $(3)$ would be a vast generalization of the Matsusaka--Ran criterion, which is the special case of the conjecture where $d=1$;   this was proposed by Debarre \cite{Debarre}, building also on previous questions of Beauville and Ran. The equivalence with $(2)$ was proposed in \cite{PP3}, in order to provide a more manageable cohomological bridge between the two conditions. Note that in $\PP^n$ there are quite familiar analogues of these equivalences; we recall this next:

(i) A formal analogy between the subvarieties of $(A, \Theta)$  in (3) and subvarieties of minimal degree in $\PP^n$ is described in \cite[\S2(6)]{PP2}: the $W_d(C)$ in Jacobians correspond to rational normal scrolls, while in the five-dimensional case the Fano surface of lines corresponds to the Veronese surface. Thus the equivalence of (1) and (3) corresponds under this formal analogy to the classification of subvarieties of minimal degree in $\PP^n$.

(ii) On the other hand, as described in 
\cite[\S2(7)]{PP2},  work in \cite{PP1} shows that condition (2) corresponds formally
 to the $2$-regularity of $\I_X$ in the sense of Castelnuovo--Mumford, for a subvariety $X \subset \PP^n$.
  Therefore, the equivalence of (2) with the other conditions is the analogue of another well-known property in projective space:  a subvariety in $\PP^n$ has minimal degree if and only if its ideal sheaf is $2$-regular.

Regarding what is known, the equivalence of $(1)$ and $(3)$ for $d=1$ is the Matsusaka--Ran criterion, while in dimension $4$ it holds for all $d$ 
by \cite{Ran}. Debarre \cite{Debarre} also shows that it holds when $(A, \Theta)$ is already assumed to be the Jacobian of a curve. 
Similarly, H\"oring  \cite[Theorem 1.2]{Hoering2} shows that it holds when $(A,\Theta)$ is assumed to be a generic intermediate Jacobian of a cubic threefold.  
The main results of \cite{PP3} are that $(2) \implies (3)$ holds for $d =1,  g-2$, and that $(2)$ implies $(1)$ apart from the reducedness assertion, which is proven in \cite{Schreieder2}.
Condition (3) is known to imply the others; see the introduction to \cite{PP3} for an explanation and references.
In particular, \conjectureref{full_conjecture} is fully known in dimension up to $4$, and for $d=1$ in general.

As mentioned above, in this paper we prove the implication $(2) \implies (3)$ in dimension $5$. This has a special significance: besides being the next case that was not known,  it is also arguably the most interesting, due to the exceptional appearance of the Fano surface of lines. 
\theoremref{conjecture} and \theoremref{main} provide the first known criteria which detect all
Jacobians of smooth genus five curves and intermediate Jacobians of cubic
threefolds at the same time.
This also strengthens the hope for obtaining the predicted list of varieties of minimal class; we expect the implication $(1) \implies (2)$ to be a formal, albeit difficult, fact.

\subsection*{Ingredients and outline of the proof of \theoremref{main}}

The proof of \theoremref{main} relies on a number of different ingredients. Here is a rough sketch of the main steps:

(1) If $V$ or $W$ is singular,  
we show in the proof of \theoremref{thm:V=sing}  that
the singular locus $\Theta^{\sing}$ of $\Theta$ is at least two-dimensional. 
Since $(A, \Theta)$ is a generalized Prym variety, and assumed to be indecomposable, a theorem of Beauville \cite{Beauville1} implies that it must be the Jacobian of a hyperelliptic curve $C$.
Since $V$ has minimal cohomology class, Debarre's theorem then implies that $V$ corresponds to $W_2(C)$.
We can therefore assume in what follows that $V$ and $W$ are smooth.

(2) If $V$ is smooth and $V+V=\Theta$ (i.e., we assume $W=V$),  then we prove in \theoremref{thm:V=W} that $(A, \Theta)$ is isomorphic to the Jacobian of a (necessarily nonhyperelliptic) curve.
The outline  is as follows.  By \cite{Schreieder2}, $V$ has property $(\mathcal P)$ with respect to itself (see \definitionref{def:propP}), and we use this to show that $V$ has a curve summand.
Therefore $\Theta$ has a curve summand and so the main result of \cite{Schreieder1} implies that $(A,\Theta)$ is isomorphic to the Jacobian of a smooth curve.

(3) If $V$ is smooth and $V-V=\Theta$ (i.e., we assume $W=-V$),  then we prove in \theoremref{T:cubics} that $(A, \Theta)$ is isomorphic to the intermediate Jacobian of a smooth cubic threefold $Y$.
 This is done by analyzing the multiplicity of $\Theta$ at the origin, using intersection theory and Schubert calculus associated with the 
 difference map $g:V\times V\rightarrow V-V$; we obtain a point of multiplicity $3$, at which stage results from \cite{Casalaina1, Casalaina2} imply the desired characterization of $(A,\Theta)$.
 Further arguments show that one can choose this isomorphism so   as to identify $V$ with the Fano surface of lines on $Y$.

(4) In \theoremref{thm:degg}, the final step of the proof, we reduce \theoremref{main} to the cases treated above.
To this end,  we consider again the difference map  
$$g:V\times V\longrightarrow V-V.$$ 
Since $V$ has minimal class, the map  $g$ is generically finite of degree $1$, $2$, $3$ or $6$.
We prove that $\deg(g)=2$ or $3$ is impossible, that $\deg(g)=1$ implies $V+V=\Theta$ and that $\deg(g)=6$ implies $V-V=\Theta$.
The surface $W$ plays  a crucial role in the argument.  We prove that 
$V-V=W-W$ and consider the pre-images $g^{-1}(W-w)$, for $w\in W$.  
The argument makes essential use of the fact from \cite{Schreieder2} (see \theoremref{thm:propP} below)  that $V$ has property $(\mathcal P)$ with respect to $W$.
The most difficult part is to rule out $\deg(g)=2$ or $3$.
Here the main strategy is to construct a curve summand of $V$ or $W$, which implies via \cite{Schreieder1} that $(A,\Theta)$ is isomorphic to the Jacobian of a smooth curve $C$;   Debarre's result \cite{Debarre} for Jacobians implies that $V$ is a translate of $W_2(C)$, which leads to a contradiction, since the difference map then has degree $1$.

\subsection*{Proof of \theoremref{conjecture}}
Once \theoremref{main} is proven, \theoremref{conjecture} is obtained  via the following argument; we recall that it suffices to treat the case of $GV$-subschemes of dimension $d=2$, since the cases $d=1$ and $3$ are already known by \cite{PP3}. 

As in \cite{PP3}, given a closed reduced subscheme $X$ of a ppav $(A,\Theta)$, we can consider the associated  \emph{theta-dual} scheme $V(X)$; set theoretically we have 
$$V(X)=\{x\in A: X-x\subseteq \Theta\}.$$
If $X$ is a geometrically nondegenerate closed $GV$-subscheme of dimension two,  we know from \cite{PP3} that $[X]= \frac{\theta^3}{3!}$, and that $V(X)$ is two-dimensional and of the same minimal class.
(In \cite{PP3}, $X$ is assumed to be pure-dimensional and reduced, but the same proofs work without these assumptions -- pure-dimensionality turns out to be automatic, see also \cite[Section 2.5]{Schreieder2}.) 
If $(A,\Theta)$ is indecomposable, then by  \cite{Schreieder2} both $X$ and $V(X)$ are irreducible and reduced.
This in turn implies $X-V(X)=\Theta$ by the nondegeneracy of $X$ and the definition of $V(X)$; see also \lemmaref{lem:deg(f)} below.
At this point, \theoremref{main} applies with $V=X$ and $W=-V(X)$ to give \theoremref{conjecture}. 

\subsection*{Conventions}
We work over the field of complex numbers $\CC$. 
A \emph{variety} is an integral separated scheme of finite type over $\CC$.
A \emph{curve} (resp.\ \emph{surface}) is a variety of dimension one (resp.\ two).
Unless mentioned otherwise, all varieties will be assumed to be proper.
A general point of a variety is a point of a Zariski open and dense subset.

\subsection*{Acknowledgements}
Part of this work was carried out during the 2015 AMS Algebraic Geometry Summer Research Institute at the University of Utah.  
The second author would like to thank Giuseppe Pareschi for conversations 
about this subject over the years, and to the University of Michigan for hospitality  during part of the preparation of the paper. 
We are grateful to Zhi Jiang, who pointed out a gap in Section \ref{subsec:deg=3} of a previous version of this paper, and a referee for useful suggestions.

\section{Preliminaries}

\subsection{Nondegeneracy}\label{scn:nondegeneracy}
We recall that, according to \cite{Ran}, a closed subscheme $X$ of dimension $d$ of an abelian variety $A$ of dimension $g$ is called \emph{nondegenerate} in $A$ if the cup product map 
\begin{equation}\label{E:NonDegCup}
\cup[X]:H^{0,d}(A)\longrightarrow H^{g-d,g}(A)
\end{equation}
is injective (hence an isomorphism). 
This definition depends only on the $d$-dimensional components of $X$, since $\omega\cup [X']=0$ for all $\omega\in H^{d,0}(A)$ if $\dim(X')<d$.
If $X$ is pure-dimensional and reduced, it is shown in \cite{Ran} that this is equivalent to the image of the Gauss map 
$$\Gamma: X^{{\rm reg}} \rightarrow \GG (d-1, g-1)$$ 
not being contained in any hyperplane via the Pl\"ucker embedding, which is further equivalent to the injectivity of the restriction map
$H^0 (A, \Omega_A^d) \rightarrow H^0 (X^{{\rm reg}}, \Omega_{X^{{\rm reg}}}^d)$, where $X^{\rm{ reg}}$ is the smooth locus of $X$.

Our main case of interest is the following:  a codimension $k$  subscheme  of a ppav $(A,\Theta)$ is nondegenerate if its cohomology class is a multiple of the minimal class $\frac{\theta^k}{k!}$, as  follows directly from the definition above.

\theoremref{conjecture} is phrased in terms of the weaker notion of \emph{geometric nondegeneracy}, which means that the kernel of the above cup product map \eqref{E:NonDegCup} 
contains no decomposable elements. 
This notion also originated in Ran's work \cite{Ran}; it is treated in the present generality in \cite{Schreieder2}. 
It turns out that $X$ is geometrically nondegenerate if and only if the reduced scheme $X^{\red}$ is.
Moreover, if $X$ is reduced, then it is geometrically nondegenerate if and only if the kernel of the restriction map
$$H^0 (A, \Omega_A^d) \longrightarrow H^0 (X^{{\rm reg}}, \Omega_{X^{{\rm reg}}}^d)$$ 
does not contain any decomposable elements, see \cite[Section 2.4]{Schreieder2}. 
However, geometrically nondegenerate $GV$-subschemes are known to represent minimal cohomology classes \cite{PP3}; it follows \emph{a posteriori} that they are nondegenerate.  

\subsection{The sum of two surfaces of minimal class}
Let $V$ and $W$ be closed subschemes of an abelian variety $A$.
The sum $V+W$ denotes the image of the addition morphism 
$$+ : V\times W\longrightarrow A;$$ 
it has the expected dimension if $\dim(V+W)=\dim(V)+\dim(W)$ or $V+W=A$. 

The sum of a geometrically nondegenerate subscheme with any other subscheme of an abelian variety has the expected dimension;  see \cite[Ch.8, ~Corollary 11]{Debarre2} for the case of subvarieties and \cite[Lemma 5]{Schreieder2} in general. 
The following well-known lemma yields a more precise answer in the case where we consider the sum of two surfaces of minimal class.

\begin{lemma} \label{lem:deg(f)}
Let $(A,\Theta)$ be a ppav of dimension $5$, and let $V,W\subseteq A$ be surfaces of minimal cohomology class, i.e.  
$[V]=[W]=\frac{\theta^{3}}{3!}$.
Then $[V+W]=\lambda\cdot \theta$ for some  $\lambda\in \{1,2,3,6\}$.
Moreover, the addition morphism 
$$f:V\times W\longrightarrow V+W$$ is generically finite of degree $ \deg(f)=\frac{ 6}{\lambda}$.
\end{lemma}
\begin{proof}
This follows from 
$$
6\cdot \theta=\frac{\theta^{3}}{3!}\star \frac{\theta^{3}}{3!}=[V]\star [W]=\deg(f)\cdot [V+W] ,
$$ 
where $\star$ denotes the Pontryagin product on the cohomology of $A$ (see e.g., \cite[Corollary 16.5.8]{BL}).  
\end{proof}

\subsection{The property $(\mathcal P)$}
Let $V$ and $W$ be subvarieties of an abelian variety $A$. 
The following definition can be found in Debarre's paper \cite{Debarre}; it originated in Ran's work \cite{Ran,Ran2}.

\begin{definition} \label{def:propP}
Let $f:V\times W\rightarrow A$ be the addition map. 
We say that $V$ has property $(\mathcal P)$ with respect to $W$ if for general $v\in V$, the only subvariety of $f^{-1}(v+W)$ which dominates $W$ via the second projection and $v+W$ via $f$ is $v\times W$.
\end{definition}

The condition in the above definition implies that for general $v\in V$, the only component of the reduced preimage $f^{-1}(v+W)^{\red}$ which dominates $W$ via the second projection and $v+W$ via $f$ is $v\times W$.
In particular, $f$ is generically finite.
Moreover, if $\deg(f)=1$, then the property $(\mathcal P)$ is trivially satisfied.
However, if $\deg(f)\geq 2$, then the condition is quite restrictive.
A guiding example is given by $W_d(C)$ inside the Jacobian $JC$ of a smooth genus $g$ curve $C$, which has property $(\mathcal P)$ with respect to $W_{g-e}(C)$ for any $e\geq d$, see \cite[Example 2.2]{Debarre}.
The following result of Debarre yields a different interpretation of this notion.

\begin{lemma}[{\cite[Lemma  2.3]{Debarre}}] \label{lem:debarre:propP:W-W}
Let $V$ and $W$ be two subvarieties of an abelian variety $A$ and let $g:V\times V \rightarrow A$ be the difference map. 
Then $W$ has property $(\mathcal P)$ with respect to $V$ if and only if, for $w\in W$ general, the only subvariety of $g^{-1}(W-w)$ which dominates both factors of $V\times V$ is the diagonal.
\end{lemma}

The following theorem, which is a special case of the main result in \cite{Schreieder2}, is one of the key ingredients of this paper.

\begin{theorem}[{\cite[Corollary 21]{Schreieder2}}] \label{thm:propP}
Let $(A,\Theta)$ be an indecomposable ppav, and let $V,W\subseteq A$ be subvarieties of dimensions $d$ and $g-d-1$ respectively, whose cohomology classes are minimal, i.e.
$$[V]=\frac{\theta^{g-d}}{(g-d)!} \,\,\,\,\,\, {\rm  and} \,\,\,\,\,\, [W]=\frac{\theta^{d+1}}{(d+1)!} \, .$$
If $V+W=\Theta $, then  $V$ has property $(\mathcal P)$ with respect to $W$ and vice versa. 
\end{theorem}

\section{Detecting Jacobians of hyperelliptic curves}
\subsection{}
The purpose of this section is to deal with the case where $V$ or $W$ in \theoremref{main} is singular.
We show that this corresponds to the case of hyperelliptic Jacobians.

\begin{theorem} \label{thm:V=sing}
Let $(A,\Theta)$ be an indecomposable ppav of dimension $5$. 
Suppose that there exist two surfaces $V,W\subseteq A$ of minimal cohomology class $\frac{\theta^3}{3!}$, 
such that $V+W=\Theta$.
If either $V$ or $W$ is singular, then there is a smooth hyperelliptic curve $C$ and an isomorphism $(A,\Theta) \cong (J(C),\Theta_C)$ which identifies $V$ with $W_2(C)$. 
\end{theorem}

\begin{proof}
By symmetry, it suffices to deal with the case where $V$ is singular.
We claim that
\begin{align} \label{eq:Vsing+V}
(V^{\sing}+W)^{\red}\subseteq \Theta^{\sing} ,
\end{align}
where $\Theta^{\sing} $ denotes the singular locus of $\Theta$ and $(V^{\sing}+W)^{\red}$ denotes the reduced scheme structure on the image $V^{\sing}+W$.
In order to prove the above inclusion, let $v_0\in V^{\sing}$ be a singular point of $V$.
Then the Zariski tangent space
\[
T_{v_0}V\subseteq T_{0}A
\]
is at least three-dimensional. 

Let us consider the Gauss map
$$
\Gamma:W^{{\rm reg}} \longrightarrow G (2, 5), \,\,\,\,\,\, w\mapsto [T_w W \subseteq T_0 A],
$$
where we see the tangent space $T_wW$ as being translated inside the tangent space of $A$ at the origin.
Since $W$ has minimal cohomology class, it is nondegenerate, and so the image of $\Gamma$ via the Pl\"ucker embedding $G(2,5)\hookrightarrow \PP^9$ is not contained in any hyperplane.
This implies that for general $w\in W$,
\[
T_{v_0}V+ T_{w}W=T_{0}A . 
\]
Hence, $v_0+w\in \Theta^{\sing}$ for general $w\in W$ and so $v_0+W\subseteq \Theta^{\sing}$. 
This proves (\ref{eq:Vsing+V}).

Since $V$ contains a singular point by assumption, (\ref{eq:Vsing+V}) implies $\dim(\Theta^{\sing})\geq 2$.   Now as $(A,\Theta)$ is five-dimensional, it is a generalized Prym variety.
Therefore, $(A,\Theta)$ indecomposable with $\dim(\Theta^{\sing})\ge 2$ implies by a result of Beauville \cite[Theorem 4.10, Theorem 5.4]{Beauville1}  that $(A,\Theta)$ is isomorphic to the Jacobian of a smooth hyperelliptic curve.
The identification of $V$ with $W_2 (C)$ follows then from Debarre's proof of the minimal class conjecture for Jacobians 
\cite[Theorem 5.1]{Debarre}.  
\end{proof}

\begin{remark}  More generally, the above arguments relate the locus of ppavs with decomposable theta divisor $\Theta=V+W$ to the Andreotti--Mayer loci $N_\ell$; i.e.,   
the loci  of ppavs $(A,\Theta)$ of dimension $g$ with $\dim ( \Theta^{\sing}) \ge \ell$.  
Indeed, if $(A,\Theta)$ is  a  ppav of dimension $g$, 
admitting  subvarieties $V,W\subseteq A$ with $V+W=\Theta$, 
then $V$ and $W$ are nondegenerate by \cite[Theorem 1]{Schreieder2}.
If $V$ is singular the arguments above therefore show 
$$
(A,\Theta)\in N_{g-1-\dim V}.
$$
In particular, if $V$ is a singular surface, then $(A,\Theta)\in N_{g-3}$, and if $V$ is a singular threefold, then $(A,\Theta)\in N_{g-4}$.  Andreotti and Mayer have shown that  the closure of the Jacobian locus $ \bar J_g\subseteq \mathcal A_g$ is an irreducible component of $N_{g-4}$ (although it is  never equal to $N_{g-4}$ for $g\ge 4$ \cite{Beauville1}).  Moreover, for indecomposable Jacobians and generalized Prym varieties,   $(A,\Theta)\in N_{g-3}$ implies that $(A,\Theta)$ is the Jacobian of a hyperelliptic curve \cite{Beauville1}.  
\end{remark}

\section{Detecting Jacobians of nonhyperelliptic curves}

\subsection{}
In this section we prove the following special case of \theoremref{main}, which isolates the case of nonhyperelliptic Jacobians.

\begin{theorem}\label{thm:V=W}
Let $V$ be a smooth surface in a five-dimensional indecomposable ppav $(A,\Theta)$, of minimal cohomology class
$[V]=\frac{\theta^3}{3!}$. 
If  $V+ V = \Theta$, then there is a smooth nonhyperelliptic curve $C$ and an isomorphism $(A,\Theta)\cong (JC,\Theta_C)$ which identifies $V$ with $W_2(C)$.
\end{theorem}

\begin{proof}
What follows uses ideas from \cite[Theorem 5]{Ran}.
Let $f:V\times V \rightarrow \Theta$ be the addition morphism.
Since $V$ has minimal cohomology class, $f$ is generically finite of degree $6$ by \lemmaref{lem:deg(f)}.

Let us pick a general point $v\in V$.
The reduced preimage of $v+V$ decomposes as
\[
f^{-1}(v+V)^{\red}=(v\times V)\cup (V\times v) \cup R \cup Q ,
\] 
where $f(Q)\subsetneq v+V$ while each component of $R$ dominates $v+V$ via $f$. 
The restriction of $f$ to $R$ has degree $4$, because for general $v_2\in V$, $f^{-1}(v+v_2)$ is given by six distinct points. 
In particular, we can pick a two dimensional component $R'\subseteq R$ which dominates $v+V$ via $f$. 

By \theoremref{thm:propP}, $V$ has property $(\mathcal P)$ with respect to itself.
Hence, ${\rm pr}_2(R')\subsetneq V$.
After swapping the factors of $V\times V$, the same argument shows ${\rm pr}_1(R')\subsetneq V$.
Since $R'$ is an irreducible surface in $V\times V$, these observations identify $R'$ with  the product ${\rm pr}_1(R')\times {\rm pr}_2(R')$ in $V\times V$:
\[
R'={\rm pr}_1(R')\times {\rm pr}_2(R') .
\] 
Applying $f$ shows that $V={\rm pr}_1(R')+{\rm pr}_2(R')$ has a curve summand, hence so does $\Theta$.
It therefore follows from \cite{Schreieder1} that there is a smooth curve $C$ and an isomorphism $(A,\Theta)\cong(JC,\Theta_C)$.

Since $V$ has minimal cohomology class, it again follows from \cite[Theorem 5.1]{Debarre} that we can choose this isomorphism in such a way that it identifies $V$ with $W_2(C)$.
Since $V$ is smooth, $C$ must be nonhyperelliptic, which concludes the proof. 
\end{proof}

\section{Detecting intermediate Jacobians of cubic threefolds}

\subsection{}
The aim of this section is to prove the following result, which is the next step towards the proof of \theoremref{main}.

\begin{theorem}\label{T:cubics}
Let $V\subset (A,\Theta)$ be a smooth surface in an indecomposable ppav of dimension $5$, of minimal   cohomology class $[V]=\frac{\theta^3}{3!}$. 
If  $V- V = \Theta$, then there is a smooth cubic threefold $Y$ and an isomorphism $(A,\Theta)\cong (JY,\Theta_Y)$ which identifies $V$ with the Fano surface of lines on $Y$.
\end{theorem}

The rest of this section is devoted to the proof of \theoremref{T:cubics}; the notation will always be that in the statement of the theorem.  

\subsection{The difference map}
We start by considering the natural difference map
$$g: V \times V \longrightarrow V - V=\Theta \subset A.$$
Since $V$ has minimal class, $\deg(g)=6$ by \lemmaref{lem:deg(f)}. 

\begin{lemma}\label{scheme}
The fiber $g^{-1}(0)$ is scheme-theoretically isomorphic to the diagonal $\Delta \subseteq V\times V$, which is isomorphic to $V$.
\end{lemma}

\begin{proof}
This is clear set-theoretically. On the other hand, if we fix a point $x \in V$, the differential of the difference map at 
$(x, x)$ is  given by vector difference, i.e. 
$$ dg_{(x,x)}: T_{(x,x)} V \times V \longrightarrow T_0 A, \,\,\,\,\,\, (v, w) \mapsto v -w$$
where we  consider $T_x V$ inside $T_0 A$ by translation. It follows that ${\rm Ker} ~dg_{(x,x)}$ is two-dimensional, and so the scheme-theoretic fiber is smooth.
\end{proof}

We can consider then the normal bundle $N$ to $g^{-1}(0)$ in $V\times V$, and we have
$$N  = N_{\Delta/V\times V} \simeq TV.$$
For Segre classes, this means that the coefficient of $[V]$ in $s(\Delta, V\times V)$ is 
$$s_2 (TV) = c_1(V)^2 - c_2 (V).$$
Recall also (see \cite[\S4.3]{Fulton}) that $s(0, \Theta) = {\rm mult}_0 \Theta$. From the general push-forward formula \cite[Proposition 4.2(a)]{Fulton}, combined with \lemmaref{scheme}, we immediately obtain:

\begin{lemma}\label{formula}
The multiplicity of $\Theta=V -V$ at the origin satisfies
$$\deg (g) \cdot {\rm mult}_0 \Theta = c_1(V)^2 - c_2 (V).$$
\end{lemma}

\subsection{The Euler characteristic of $V$}
General results imply the following lower bound:

\begin{proposition}\label{P:ChiEst}
The holomorphic Euler characteristic of $V$ satisfies
$$
\chi(\mathscr O_V)\ge 4.
$$
\end{proposition}

\begin{proof}
Since $V$ is smooth and nondegenerate in $A$, it is canonically polarized; see e.g.~\cite[p.98--101]{Debarre2}.  Being in particular of general type,  it then follows from \cite{Beauville3} that
$
p_g(V)\ge 2q(V)-4
$.
In other words $q(V) \le \frac{1}{2}p_g(V) +2$.  We then have that
$$
\chi(\mathscr O_V) =1 -q (V)+p_g (V) \ge  1-\left(\frac{1}{2}p_g (V) +2\right) +p_g (V)  =\frac{p_g (V) }{2}-1.
$$
Since $V$ is nondegenerate, we have $p_g (V)\ge \binom{5}{2} = 10$ (in fact, equality holds by \cite[Corollary 15]{Schreieder2}), from which it follows that  $\chi(\mathscr O_V)\ge 4$.  
\end{proof}

\subsection{The Gauss map}\label{S:G-Map}
Recall that the inclusion of $V$ in $A$ provides a Gauss map
$$\Gamma: V \longrightarrow \GG (1, 4), \,\,\,\,\,\,s\mapsto [\PP T_s V \subseteq \PP T_0 A],$$
where we see all the tangent spaces to points of $V$ as being translated inside the tangent space of $A$ at the origin.   Since $V$ is smooth and nondegenerate in $A$, the Gauss map is finite; see e.g.~\cite[Proposition~8.12, Corollary~15, p.101]{Debarre2}.

\begin{remark}\label{surface_class}
In the cohomology ring of $\GG (1, 4)$ we have the identification
\begin{equation}\label{E:Gam*}
\Gamma_* [V] = \big(c_1(V)^2 - c_2 (V) \big) \cdot \sigma_2^\vee + c_2(V) \cdot \sigma_{1,1}^\vee,
\end{equation}
where $(-)^\vee$ denotes the Poincar\'e dual; concretely, $\sigma_2^\vee=\sigma_{31}$, $\sigma_{1,1}^\vee=\sigma_{22}$.
Indeed, by definition, if $T$ is the tautological bundle on $\GG(1, 4)$, then 
$TV \simeq N \simeq \Gamma^*T$.  
It follows that the Segre class of $TV$ satisfies 
$$s(TV) = \Gamma^* (1 + \sigma_1 + \sigma_2 + \cdots) \in H^{*} ( V, \ZZ).$$
In particular we have that $s_2 (TV) = \Gamma^* \sigma_2$, and so by the projection formula we get 
$$\Gamma_* [V] \cdot \sigma_2 =  s_2 (TV) = c_1(V)^2 - c_2 (V).$$
This gives the first coefficient. The second can be obtained similarly, noting that $c_2 (V) = \Gamma^* \sigma_{1,1}$.
\end{remark}

Considering now the difference map $g:V\times V\to V-V\subseteq A$, 
blowing up $0$ in $\Theta = V-V$ and $\Delta$ in $V\times V$, we obtain a surjective morphism on exceptional divisors
$$
dg:\mathbf PN \longrightarrow \mathbf PC_0\Theta \subseteq \mathbf PT_0A,
$$
which is generically finite, and induced by the differential of $g$.   Here $C_0\Theta$ is the tangent cone to $\Theta$ at the origin.  Considering $dg$ on fibers  $\mathbf P(N_v)$ of the projectivized normal bundle, there is an induced map $V\to \mathbf G(1,4)$, which agrees with the Gauss map.  More precisely, 
under the identification of the bundles
$N  = N_{\Delta/V\times V} \simeq T_V$,
the assignments
$$
v\mapsto \mathbf Pdg_v(N_v)\subseteq \mathbf PT_0A
$$
$$
v\mapsto \mathbf PT_v V\subseteq \mathbf PT_vA=\mathbf PT_0A
$$
agree.  
It follows that if we denote by $F$ be the Fano variety of lines in $\mathbf PC_0\Theta \subseteq \mathbf P^{4}$, then the Gauss map factors as 
\begin{equation}\label{E:GaussFact}
\Gamma:V\longrightarrow F\hookrightarrow \mathbf G(1,4).
\end{equation}
Moreover, the lines parameterized by $V$ cover ${\bf P}C_0\Theta$.

\subsection{The case of quadric tangent cones}
In this section, in order to understand  what would happen if 
$$Q = \PP C_0 \Theta \subseteq \PP^4$$ 
were  a quadric hypersurface, we   recall some elementary facts about Fano varieties of lines on quadric threefolds.

\begin{proposition}\label{P:FanoQuad}
Let $Q\subseteq \mathbf P^4$ be a quadric threefold.  
\begin{enumerate}
\item If $Q$ is singular, then every irreducible component of the Fano variety $F$ of lines on $Q$ is contained in some Schubert 
variety $\Sigma_1 \subset \GG(1,4)$. In other words, there exists a plane $\PP^2 \subset \PP^4$ meeting each line parametrized 
by the respective component.

\item If $Q$ is smooth, then the Fano variety $F$ of lines on $Q$ satisfies $F \simeq \PP^3$. Moreover, if we 
denote by $\gamma : F \hookrightarrow \GG(1,4)$ the inclusion map, and by $H$ a hyperplane in $\PP^3$, via this isomorphism 
we have
$$\gamma_*[H]=2\sigma_{31}+2\sigma_{22}\in H^4 \big(\GG(1,4),\mathbf Z \big).$$
\end{enumerate}
\end{proposition}
\begin{proof}  
Everything besides the last calculation is folklore; we sketch the proofs for completeness, and direct the reader to  
\cite[Ch.~22]{harris} for some of the standard results on Fano varieties of linear spaces on quadrics.  

\noindent
(1) If $Q$ is a rank four quadric, then it is a cone over a smooth quadric surface.  The Fano variety of lines then consists of two irreducible components $F'$ and $F''$, one for each family of lines on the smooth quadric.  These components meet along the family of lines through the vertex.  
It is clear from this description that any line on $Q$ corresponding to a point in $F'$ meets the plane spanned by the vertex of $Q$ and a line on the smooth quadric in the ``opposite'' ruling, i.e. the ruling defining $F''$. 

If $Q$ is a rank three quadric, then it is a cone over a singular irreducible quadric surface.   Every line on $Q$ meets the line through the vertex and the singular point of the quadric surface, and in particular any plane containing it.  

If $Q$ is the union of hyperplanes, then the conclusion is obvious.

\noindent
(2) Let $Q$ be a smooth quadric threefold with Fano variety of lines $F$.  We start by recalling a standard construction of the isomorphism
$F\cong {\bf P}^3$ (see e.g.~\cite[Exercise~22.6, p.290]{harris}).  Let $V$ be a four dimensional $\CC$-vector space, and let $\Omega:V\times V \rightarrow \mathbb \CC$ be a non-degenerate skew symmetric bilinear form.  Then the set 
$$Q'=\{\Lambda\in G(2,V):\Omega(\Lambda,\Lambda)=0\}$$ is a smooth hyperplane section of the Pl\"ucker embedding of $G(2,V)$ in 
${\bf P}^5$.  As $Q'$  is a smooth quadric hypersurface, we can take $Q'=Q$.
The  identification of ${\bf P}V$ with $F$ is made as follows:  for $[v]\in {\bf P}V$, the set
$$
\ell_{[v]}=\{\Lambda\in G(2,V):  \Lambda \supseteq \langle v\rangle ,\ \Omega(\Lambda,\Lambda)=0\}\subseteq Q
$$
is a line in $Q$.  Conversely, given any line $\ell$ on $Q$, there is a unique $[v]\in {\bf P}V$ so that the ray $\langle v\rangle$ is contained in every $2$-plane corresponding to a point of $\ell$.  

Now using $\Omega$, there is an identification ${\bf P}V ={\bf P}V^\vee$.  
Therefore, for $[v]\in {\bf P}V$, we have the associated hyperplane
$$
H_{[v]}:={\bf P}\{v'\in V:\Omega(v',v)=0\} \subseteq {\bf P}V.
$$
Note that there is a bijection
$$
H_{[v]}\longrightarrow \{\ell'\subseteq Q: \ell'\cap \ell_{[v]}\ne \emptyset\}\subset F, \,\,\,\,\,\, [v']\mapsto \ell_{[v']}.
$$

Moving on to the computation of the class $\gamma_*[H]$ for a hyperplane $H\subseteq \mathbf PV$, for dimension reasons we have
$$
\gamma_*[H]=a\sigma_{31}+b\sigma_{22},
$$
where $a=\gamma_*[H]\cdot \sigma_2$ and  $b=\gamma_*[H]\cdot \sigma_{11}$.
We first claim that 
$$\gamma_*[H] \cdot \sigma_{11}=2.$$  
Indeed, assume $H=H_{[v]}$ is the set of lines on the quadric  meeting the line $\ell_{[v]}$.  Recall that $\sigma_{11}$ corresponds to lines in ${\bf P}^4$ contained in a given hyperplane.  A general such hyperplane cuts $Q$ in a smooth quadric surface, and cuts $\ell_{[v]}$ at a single point, which lies on $Q$.  Therefore, we are asking for the set of lines on a smooth quadric, through a given point, which proves the claim. Next we claim that 
$$\gamma_*[H]\cdot \sigma_2=2.$$
Assume again that $H=H_{[v]}$ is the set of lines on the quadric  meeting $\ell_{[v]}$.   
 Recall that $\sigma_{2}$ corresponds to lines in ${\bf P}^4$ meeting a given line $\ell$.  
Take $\ell$ general, and consider the $3$-plane $\langle \ell,\ell_{[v]}\rangle$.  
This cuts $Q$ in  a smooth quadric surface, that meets $\ell$ in two points.  
 Therefore, we are asking for the set of lines on a smooth quadric, that meet a given line and pass through one of two general points.  There are exactly two, one for each point.   This proves the claim.
\end{proof}

\subsection{Proof of \theoremref{T:cubics}}
We can now put together all the ingredients discussed in the previous subsections.

\subsubsection{$(A,\Theta)$ is isomorphic to the intermediate Jacobian of a smooth cubic threefold}

Our aim will be to show that
\begin{equation}\label{high}
{\rm mult}_0 \Theta = \deg \PP C_0 (\Theta) \ge 3.
\end{equation}
We can then apply \cite[Theorem 3]{Casalaina2}, \cite[Corollary 2.2.4]{Casalaina1}  in order to deduce that $(A, \Theta)$ is either the intermediate Jacobian of a 
cubic threefold, or the Jacobian of a hyperelliptic curve $C$.  In the latter case it follows from \cite[Theorem 5.1]{Debarre} that $V$ must be a translate of $W_2(C)$ or $-W_2(C)$. However, in this case $V$ is singular, which contradicts the hypothesis.

It remains then to prove (\ref{high}).    To this end, note first that since $V$ is nondegenerate as it has minimal class, $\PP C_0 \Theta$ cannot be contained in a hyperplane in $\PP^4$.   Indeed, since   the Gauss map $\Gamma$ factors through the inclusion of the Fano variety of lines in  $\PP C_0\Theta$ \eqref{E:GaussFact}, the image of the Gauss map would 
be contained in a Schubert cycle $\Sigma_1$, 
 which would precisely mean that $V$ was geometrically degenerate; see \S\ref{scn:nondegeneracy}.
 This would be  a contradiction. 

Our task is therefore to show that $Q = \PP C_0 \Theta$ cannot be a quadric; assuming that it is, we will aim for a contradiction. 
If $Q$ is a singular quadric, then  by \propositionref{P:FanoQuad} every irreducible component of the Fano variety $F$ is contained in some  $\Sigma_1 \subset \GG(1, 4)$, again contradicting the nondegeneracy of $V$.

Suppose finally that $Q$ is a smooth quadric. Since $\deg(g)=6$ by \lemmaref{lem:deg(f)}, then according to \lemmaref{formula}   we have
\begin{equation}\label{E:s2}
c_1(V)^2 - c_2 (V) = \deg (g) \cdot \operatorname{mult}_0\Theta =12.
\end{equation}
At the same time, under the standard isomorphism $F\cong {\bf P}^3$ of the Fano variety of lines on a smooth quadric threefold with projective $3$-space,   and the factorization of the Gauss map
$$
V\longrightarrow F\hookrightarrow \mathbf G(1,4),
$$
we must have that the push-forward of the class of  $[V]$ to  $F$ is equal to  $n[H]$ for some $n$, where $H$ is a hyperplane in 
${\bf P}^3$.   Let $\gamma:F\hookrightarrow \mathbf G(1,4)$ be the inclusion.  
Since by \propositionref{P:FanoQuad} we have
$\gamma_*[H]=2\sigma_{31}+2 \sigma_{22}$, 
it follows that  
$$\Gamma_*[V]=\gamma_*n[H]=2n\sigma_{31}+2n \sigma_{22}.$$
Therefore, in particular, $\sigma_{11} \cdot \Gamma_*[V] = \sigma_2 \cdot \Gamma_*[V]$.  Combining  \eqref{E:Gam*} and  \eqref{E:s2}, we have  that $\sigma_2 \cdot \Gamma_*[V]=c_1(V)^2-c_2(V)=12$.  
From  \eqref{E:Gam*}  we therefore also have $c_2(V)=\sigma_{11} \cdot \Gamma_*[V]=12$.  
We conclude that $ c_1(V)^2=24$.     Noether's formula then implies
$\chi(\mathscr O_V)=3$,  contradicting  \propositionref{P:ChiEst}.

\subsubsection{$V$ is isomorphic to the Fano surface} \label{subsec:FcongS}
We have established that there is an isomorphism $(A,\Theta)\cong (JY,\Theta_Y)$ for some cubic threefold $Y$. 
Let $F$ be the Fano surface of lines on $Y$; we now wish to show that $V\cong F$.  

We start by recalling a few facts about $F$. First, from Clemens--Griffiths \cite{CG}  we have that $JY\cong \operatorname{Alb}(F)$, the Albanese embedding of $F$ has class $[F]=\theta^3/3!$, and $F-F=\Theta$.   
 In other words, $F\subseteq A$ also satisfies the hypotheses of the theorem.  

From the (degree $6$)  difference map $g_V:V\times V\to \Theta_Y$, and the fact established by Mumford \cite{Mumford} that 
${\bf P}C_0\Theta_Y\cong Y$,  we obtain that the Gauss map $\Gamma_V$ factors \eqref{E:GaussFact} as 
$$
\Gamma_V:V\longrightarrow F\stackrel{i}{\hookrightarrow} \mathbf G(1,4).
$$
Since, as $V$ is smooth and nondegenerate, we have seen earlier in \S\ref{S:G-Map} that $\Gamma_V$ is  finite, it follows that the image of $V$ is $F$.  Thus, in $H^4(\mathbf G(1,4),\mathbf Z)$ we have 
$
\Gamma_{V*}[V]=n_V i_{*}[F]
$
for some number $n_V$, so that 
$$
\sigma_2 \cdot \Gamma_{V*}[V]=n_V\sigma_2 \cdot i_{*}[F].
$$
 In the special case that $V=F$, we know that the degree of the Gauss map is $1$ (i.e, $n_F=1$).   In particular, 
 $$\sigma_2 \cdot i_*[F]=\sigma_2 \cdot \Gamma_{F*}[F]=c_1(F)^2-c_2(F)=\deg (g_F)\cdot \operatorname{mult}_0\Theta_Y= 6 \cdot 3.$$   
 Using Remark \ref{surface_class}, the same computation shows that $\sigma_2 \cdot \Gamma_{V*}[V]=18$, as well.  Therefore $n_V=1$, 
 and $\Gamma_*[V]=i_*[F]$.  In particular, the degree of the finite  Gauss map $\Gamma:V\to F$ is $1$, and so $V\cong F$
 since $F$ is smooth.

\subsubsection{$V$ can be identified with a translate of the Fano surface} \label{subsubsec:S=F}
Here we show that there is an isomorphism $(A,\Theta)\cong (JY,\Theta_Y)$ which induces the isomorphism between 
$V$ and $F$ constructed in \S\ref{subsec:FcongS} above. This will finish the proof of \theoremref{T:cubics}.

Consider the inclusion of $V\hookrightarrow A$.   From the universal property of the Albanese, we obtain a commutative diagram
$$
\begin{tikzcd}
  V \arrow{d}[swap]{\cong} \arrow[hook]{r}& A \\
F \arrow[hook]{d}[swap]{\operatorname{Alb}_F}& \\
\operatorname{Alb}(F)=JY  \arrow[dotted]{ruu}{\xi}&
  \end{tikzcd} 
$$
In other words, 
since $JY$ is isomorphic to the Albanese variety of $V\cong F$ \cite{CG}, there is a morphism
$
\xi:JY\rightarrow A 
$,
whose restriction to $F\subset JY$ yields an isomorphism 
$$\xi |_{F}:F\stackrel{\sim} \longrightarrow V.$$
Since $V$ generates $A$, $\xi$ is surjective, and  hence an isogeny for dimension reasons.  Our goal is to show that the isogeny $\xi$ is an isomorphism. 

To this end, observe that 
$$
\xi(\Theta_Y)=\xi(F-F)=V-V=\Theta .
$$
In fact, $\xi |_{\Theta_Y}:\Theta_Y\rightarrow \Theta $ is generically finite of degree one, because  $\xi |_{F}$
is an isomorphism and the difference maps $V\times V\rightarrow V-V$ and $F\times F\rightarrow F-F$ both have degree $6$.  It follows that $\xi |_{\Theta_Y}:\Theta_Y\rightarrow \Theta $ is birational.
Since $JY$ is the Albanese variety of $\Theta_Y$ and $A$ is the Albanese variety of $\Theta$, we obtain an induced morphism
$$
\alb(\xi |_{\Theta_Y}):JY\longrightarrow A 
$$
which coincides with $\xi$, since its restriction to $\Theta_Y$ has that property.
On the other hand, $\alb(\xi |_{\Theta_Y})$ is an isomorphism, because $\xi |_{\Theta_Y}$ is birational.
We conclude that 
$$\xi:(JY,\Theta_Y)\longrightarrow (A,\Theta)$$ 
is an isomorphism of ppavs which identifies $F$ with $V$.

\begin{remark} \label{rem:Lombardi-Tirabassi}
If in the proof of \theoremref{T:cubics} one assumes additionally that $\mathcal I_V(\Theta)$ is a $GV$-sheaf, then the conclusion of \S\ref{subsubsec:S=F} follows immediately from the paper \cite{LT} of  Lombardi and Tirabassi, where it is shown that the inclusion of a $GV$-subvariety does not factor through any nontrivial isogeny.
\end{remark}

\section{Reduction to the known cases}

\subsection{} \label{subsec:thm:degg}

The purpose of this section is to prove the following result, \theoremref{thm:degg}, which reduces \theoremref{main} to \theoremref{thm:V=sing}, \theoremref{thm:V=W} and \theoremref{T:cubics} above.

\begin{theorem} \label{thm:degg}
Let $(A, \Theta)$ be an indecomposable ppav of dimension $5$, and let $V, W\subset A$ be smooth surfaces such that $[V] = [W] = \frac{\theta^3}{3!}$ and $V  + W = \Theta$. 
Let 
$$g:V\times V\longrightarrow V-V \subset A$$ 
denote the difference map. Then $\deg(g)\in\{1,6\}$, and moreover,

\begin{enumerate}
\item if $\deg(g)=1$, then $V+V$ is a translate of $\Theta$.
\item if $\deg(g)=6$, then $V-V$ is a translate of $\Theta$. 
\end{enumerate}
\end{theorem}

\noindent
The rest of the section is devoted to the proof of this theorem; we note that the case where $V$ is a translate of $-W$ is immediate by \lemmaref{lem:deg(f)}, and that we can use \theoremref{thm:V=W} to deal with the cases where $V$ is a translate of $W$.

\subsection{}
We begin by recalling that by \lemmaref{lem:deg(f)},
$$
\deg(g)\in \{1,2,3,6\} .
$$
If $\deg(g)=6$, then $[V-V]=[\Theta]$ and so that statement of \theoremref{thm:degg} is immediate. 
The other cases are much more difficult, and will be treated in the following subsections respectively. 
Our notation will always be that of \theoremref{thm:degg}. 

Before we turn to our case by case study, let us first prove the following general lemma, which we will use repeatedly.

\begin{lemma} \label{lem:V-V=W-W}
Under the assumptions of \theoremref{thm:degg}, we have $V-V=W-W$.
\end{lemma}

\begin{proof}
If $(A,\Theta)$ is isomorphic to the Jacobian of a smooth curve $C$, then $V$ and $W$ are translates of $\pm W_2(C)$ 
by \cite{Debarre}, and so the lemma holds.
We may therefore assume in what follows that $(A,\Theta)$ is not isomorphic to a Jacobian.

Since $V$ and $W$ have minimal class, the addition morphism $f:V\times W\longrightarrow \Theta$ has degree $6$.
For a general point $v\in V$, the reduced preimage of $v+W$ decomposes as
\[
f^{-1}(v+W)^{\red}=(v\times W) \cup R \cup Q ,
\] 
where $f(Q)\subsetneq v+W$ and each component of $R$ dominates $v+W$ via $f$. 
For general $w\in W$, $f^{-1}(v+w)$ is a union of $6$ distinct points and so $f|_R:R\longrightarrow v+W$ has degree $5$.
In particular, $R$ is nonempty.
By \theoremref{thm:propP}, $V$ has property $(\mathcal P)$ with respect to $W$, which means that ${\rm pr}_2(R)\subsetneq W$.  
If $R$ contains a component $R'$ such that ${\rm pr}_2(R')$ is a point, then $R'=V\times w$.
Applying $f$ shows that $V$ is a translate of $W$.
It then follows from \theoremref{thm:V=W} that $(A,\Theta)$ is isomorphic to the Jacobian of a smooth curve, which contradicts our assumption.

We may therefore assume that there is a component $R'$ of $R$  whose image under the second projection is a curve 
$${\rm pr}_2(R')=C_v\subset W.$$
If ${\rm pr}_1(R')$ is also a curve, then $W$ (and hence $\Theta$) has a curve summand and so $(A,\Theta)$ is the Jacobian of a smooth curve by \cite{Schreieder1}.
We may thus assume that ${\rm pr}_1(R')=V$.
This means that for general $v_2\in V$ there are points $w_1,w_2\in W$ with
\[
v_2+w_2=v+w_1.
\] 
In particular $v-v_2\in W-W$.
Since we can make this argument for each general $v$, and $v_2$ was chosen generally,  the containment holds for general $(v,v_2)\in V\times V$, and it follows that $V-V\subseteq W-W$.
The lemma follows now from the irreducibility, reducedness and non-degeneracy of $V$ and $W$, which implies that $V-V$ and $W-W$ are both irreducible, reduced and of the same dimension; for the latter, see for instance \lemmaref{lem:deg(f)}.
\end{proof}

\subsection{The case $\deg(g)=1$} \label{subsec:deg=1'} 
Our goal will be to show that either $V$ has a curve summand, or that  $\pm V$ is a translate of $W$.   
In the latter case, if $V$ is a translate of $W$, we can use \theoremref{thm:V=W}  to conclude, and if $V$ is a translate of $-W$, the assertion is immediate by \lemmaref{lem:deg(f)}.  
In the former case,   if we succeed in showing that  $V$ has a curve summand, then it will follow that $\Theta$ has a curve summand and so $(A,\Theta)$ is isomorphic to the Jacobian of a smooth curve $C$ by \cite{Schreieder1}.
Since $V$ has minimal class, one can by \cite{Debarre} find an isomorphism $(A,\Theta)\cong (JC,\Theta_C)$ that identifies $V$ with $W_2(C)$.
In particular, $V+V=\Theta$ because $W_2(C)+W_2(C)=\Theta_C$.  

Therefore, we now proceed to show that either  $V$ has a curve summand, or that  $\pm V$ is a translate of $W$.
If $\deg (g)=1$, then  $g$ is birational.  
Consequently, for general $w\in W$, the reduced preimage
$$
g^{-1}(W-w)^{\red}\subseteq V\times V 
$$
 contains a component $R_w$ that  is birational to $W$.  
 For $i=1,2$, this gives rise to a rational map
\[
\varphi_{i,w}:W\stackrel{\sim}\dashrightarrow R_w \stackrel{{\rm pr}_i}\longrightarrow V ,
\]
which depends on $w\in W$.

 We want to study the maps $\varphi_{i,w}$ as $w$ varies.  To this end, 
for $i=1,2$ consider the diagram
$$
\begin{tikzcd}
V\times V  \arrow{d}[swap]{\operatorname{pr}_i} 
& \arrow[dashed]{l}[swap]{g^{-1}} 
V-V=W-W & W\times W \arrow{l} \arrow[dashed]{lld}{\Phi_i} \\
V& 
\end{tikzcd}
$$
where $\Phi_i$ is defined to be the rational map obtained via composition.  Since $W$ is smooth and $V\subseteq A$, we have that $\Phi_i$ is in fact a morphism.  Moreover, $\Phi_i$, viewed as a morphism to $A$,  factors through the Albanese $\operatorname{Alb}(W)\times \operatorname{Alb}(W)$.   Since, after translation,  the induced map $\operatorname{Alb}(W)\times \operatorname{Alb}(W)\to A$ decomposes as a sum of homomorphisms, we may write
$$
\Phi_i=\varphi_i+c_i:W\times W\longrightarrow V
$$
with $\varphi_i,c_i:W\to A$.  
In particular, 
\begin{align} \label{eq:im(c_i)'}
V=\im (\varphi_i)+\im (c_i)
\end{align}
for $i=1,2$.  By construction, for general $w\in W$  we have $\varphi_{i,w}(w')=\varphi_i(w')+c_i(w)$ for all $w'\in W$.  

The first  claim is that either $\im(\varphi_1)$ or $\im(\varphi_2)$ has dimension $\leq 1$.   Indeed, for  general $w\in W$, 
\begin{equation}\label{E:Rw}
\varphi_i(W)+c_i(w)=\Phi_i(W,w)=\operatorname{pr}_i(R_w).
\end{equation}
By \theoremref{thm:propP}, $W$ has property $(\mathcal P)$ with respect to $V$, and 
this implies by \lemmaref{lem:debarre:propP:W-W} that 
the only subvariety of $g^{-1}(W-w)$ that  dominates both factors of $V\times V$ is the diagonal.  In particular, $R_w$ does not dominate both factors, so that by \eqref{E:Rw} one of the  images $\varphi_i(W)$ must not be a  surface,  establishing the claim.

Now we are ready to conclude.  If $\im(\varphi_1)$ or $\im(\varphi_2)$ is a curve, then $V$ has a curve summand by (\ref{eq:im(c_i)'}).
If $\im(\varphi_1)$ or $\im(\varphi_2)$ is a point, then by the construction of $\varphi_i$, we have that $g^{-1}(W-w)^{\operatorname{red}}$ is of the form $v\times V$ or $V\times v$ for some point $v \in  V$. Applying $g$ to this shows that $\pm V$ is a translate of $ W$.    
This completes the proof of \theoremref{thm:degg} in the case where $\deg(g)=1$.

\subsection{Ruling out $\deg(g)=2$} \label{subsec:deg=2'}
In this section we will derive a contradiction from the assumption that $\deg(g)=2$.   The proof is somewhat lengthy, and we break it  up into several parts.

\subsubsection{Some preliminary observations}
We start by explaining why we may assume:
\begin{itemize}
\item $(A,\Theta)$ is not isomorphic to the  Jacobian of a smooth curve.
\item $V$ and $W$ do not admit curve summands.
\item $V$ is not a translate of $\pm W$.
\end{itemize}
Indeed,  if  $(A,\Theta)$ were isomorphic to the Jacobian of a smooth curve $C$, then 
 $C$ would have to be nonhyperelliptic since by \cite{Debarre}, $V=W_2(C)$ is assumed to be smooth, and then the  difference map 
$g:W_2(C)\times W_2(C)\rightarrow W_2(C)-W_2(C)$
 would  have degree $1$.
Consequently, since $(A,\Theta)$ can be assumed not to be a Jacobian, and $\Theta=V+W$, by \cite{Schreieder1} we may further assume that $V$ and $W$ have  no curve summand.
Finally, as mentioned previously, we may by \lemmaref{lem:deg(f)} and \theoremref{thm:V=W} assume that $V$ is not a translate of $\pm W$.

Going forward, the  main observation is the following.  For general $v\in V$, 
\[
g^{-1}(v-V)^{\red}=(v\times V)\cup R_v\cup Q ,
\]
where $g( Q)\subsetneq v-V$ and each component of $R_v$ dominates $v-V$ via $g$.\footnote{For instance, the diagonal $\Delta_V$ is contained in $Q$.}
Since $g^{-1}(v-v_2)$ is given by two distinct points if $v_2\in V$ is general, $R_v$ is irreducible and birational to $V$.
This allows us to consider the rational map
\begin{equation}\label{E:psiv}
\psi_{i,v}:V\stackrel{\sim}\dashrightarrow R_v \stackrel{{\rm pr}_i}\longrightarrow V .
\end{equation}

\subsubsection{The global description of the $\psi_{i,v}$}\label{S:GDPsi'} We want to study the maps $\psi_{i,v}$ as $v$ varries.  We will use this to arrive at the main goal of this subsection,  the description of $V$ as the image of the sum of maps in \eqref{eq:im(d_i)'}, below.

To this end, consider the map   $$(g\times 1_V):V\times V\times V\longrightarrow  (V-V)\times V$$ 
$$
(v_1,v_2,v)\mapsto (v_1-v_2,v)
$$
 and let $\mathscr V\subseteq  (V-V)\times V$ be given by the image  $(g\times 1_V)(\operatorname{pr}_{13}^{-1}(\Delta_V))$, where $\Delta_{V}$ denotes the diagonal in $V\times V$.   The fiber of $\mathscr V$ over a point $v\in V$ via the last projection is given by $\mathscr V_v^{\operatorname{red}}=(v-V)\times v$.   In other words, $\mathscr V$ is the family of $v-V$.  

The points of   $\mathscr G:=(g\times 1_V)^{-1}(\mathscr V)$ can be described set theoretically as 
$$
\mathscr G=(g\times 1_V)^{-1}(\mathscr V)=\{(v_1,v_2,v)\in V\times V\times V:  v_1-v_2\in v-V\}.
$$
By construction, the fiber of $\mathscr G$ over a point $v\in V$ via the last projection is given by $\mathscr G_v^{\operatorname{red}}=g^{-1}(v-V)^{\operatorname{red}}\times v$.  In other words, $\mathscr G$ is the family of $g^{-1}(v-V)^{\operatorname{red}}$.  

Clearly  $\operatorname{pr}_{13}^{-1}(\Delta_V)$ is one component of $\mathscr G$. We also know that $\mathscr G\to \mathscr V $ has degree two generically and so there is a unique second component 
$$\mathscr R \subseteq \mathscr G
\subseteq V\times V\times V
$$ with the property that $g\times 1_V$ restricted to $\mathscr R$ yields a dominant map $\mathscr R\to \mathscr V$. 
By construction, the fiber of $\mathscr R$ over a point $v\in V$ via  the last projection satisfies 
  $$(\mathscr R_v)^{\operatorname{red}}=R_v\times v;$$ 
i.e.,  $\mathscr R$ is the family of $R_v$.  

We next claim that $\mathscr R$ is birational to $V\times V$.  
Indeed,  consider the map over $V$
$$
F:\mathscr R \longrightarrow A \times V,
$$
$$
F(v_1,v_2,v)=(v-(v_1-v_2),v).
$$
From the definition of $\mathscr G$, the image of this map lies in $V\times V$.  Restricted to general fibers, this map $F$ is the birational map between  $R_v\times v $ and  $ V\times v$ mentioned above, and therefore $F$ is generically one-to-one, and so $\mathscr R$ is birational onto $V\times V$.

Finally, for $i=1,2$, define the map $\Psi_i$ as the composition 
$$
\Psi_i:V\times V\stackrel{F^{-1}}{\dashrightarrow} \mathscr R \stackrel{\operatorname{pr}_i}{\longrightarrow} V
$$
where $\operatorname{pr}_i$ is the projection onto the $i$-th factor.  

As in \S \ref{subsec:deg=1'}, we may use the Albanese to conclude that 
$$
\Psi_i=\psi_i+d_i:V\times V\longrightarrow V
$$
with $\psi_i,d_i:V\to A$.  
In particular, 
\begin{align} \label{eq:im(d_i)'}
V=\im (\psi_i)+\im (d_i)
\end{align}
for $i=1,2$.  By construction, for general $v\in V$  we have $\psi_{i,v}(v')=\psi_i(v')+d_i(v)$ for all $v'\in V$.

\subsubsection{Reducing to the case $V=-V$}  The purpose of this subsection is to use \eqref{eq:im(d_i)'} to reduce  to the case
\begin{align} \label{eq:V=-V'}
V=-V .
\end{align} 
The first  claim in this direction is that either $\im(\psi_1)$ or $\im(\psi_2)$ has dimension $\leq 1$.  
Indeed, since $V$ is geometrically nondegenerate, 
we know that the addition map of $V$ with any subvariety has the expected dimension.  Consequently, from \eqref{eq:im(d_i)'},  
if  $\psi_i(V)$ were a surface, and hence a translate of $V$, then $d_i$ would have to be constant.  If both the maps  $d_i$ were constant, then we would have  that
$$
R_v=\{(\psi_{v,1}(v'),\psi_{v,2}(v')) \mid  v'\in V\}=\{(\psi_{1}(v')+d_1,\psi_{2}(v')+d_2) \mid  v'\in V\}
$$
would be independent of $v$.  But this is a contradiction  since   $g(R_v)=v-V$, and the $v-V$ move (for instance since $V-V$ is of dimension $4$  by the geometric nondegeneracy of $V$).  This establishes the claim  that either $\im(\psi_1)$ or $\im(\psi_2)$ has dimension $\leq 1$.

Now, if the dimension of the image of $\psi_i$ is one, then by \eqref{eq:im(d_i)'}, $V$ would have a curve summand, which we are assuming is not the case.  Therefore, one of the $\psi_i$ is constant.  

If $\psi_1$ is  constant, then (as at the end of  \S \ref{subsec:deg=1'}) $R_v=v'\times V$ for some $v'\neq v$ (here we take $v'\ne v$ since the case $v'=v$ corresponds to the component  of $g^{-1}(v-V)^{\operatorname{red}}$ that we specifically left out of $R_v$ in the definition). 
This cannot happen  because $v-V=v'-V$ implies $v-V-W=v'-V-W$, and hence $v=v'$, since $-V-W$ is a translate of 
$\Theta$.
On the other hand, if $\psi_2$ is constant, then  $R_v=V\times v'$.
This  implies $v-V=V+v'$ for some $v'\in V$. 
Up to replacing $V$ by a translate, we may therefore assume $V=-V$.

\subsubsection{The case $V=-V$}  We have now reduced to the case $V=-V$.  
This subsection is devoted to arriving at a contradiction under this  additional assumption.

To do this, we turn our attention to the structure of the $g^{-1}(W-w)$ for $w\in W$. 
By \theoremref{thm:propP}, $W$ has property $(\mathcal P)$ with respect to $V$.
Moreover,  by \lemmaref{lem:V-V=W-W},
\[
V-V=W-W .
\]
By \lemmaref{lem:debarre:propP:W-W}, it follows that for $w\in W$ general,
\[
g^{-1}(W-w)^{\red}=R_w\cup Q_w
\]
with $R_w$ and $Q_w$ satisfying:   $g(Q_w)\subsetneq W-w$,  every component of $R_w$ dominates $W-w$ via $g$,
and  no component of $R_w$ dominates both factors of $V\times V$.
The restriction of $g$ to $R_w$ has degree $2$.
Moreover, $R_w$ is anti-symmetric in the following sense: if we let $\sigma:V\times V\to V\times V$ be the map interchanging the factors, then  since $V=-V$,  we have that $R_w=-\sigma(R_w)\subseteq V\times V$.

The main claim, proven in \lemmaref{lem:R_w=irred'} below, is that   $R_w$ is irreducible.  We assume this for the moment, and complete the argument.  
We immediately obtain from irreducibility that $R_w$ is not mapped  via either  projection surjectively onto $V$, since then by anti-symmetry, it would  be mapped surjectively onto both factors, which we have shown is not possible.  Similarly, we obtain that $R_w$ is not mapped  via either  projection to a point. 
 Indeed, without loss of generality, if  $\operatorname{pr}_1(R_{w})=v$, then $R_{w}=v\times V$, and so by anti-symmetry   we have  $R_{w}=(V\times (-v))\cup (v\times V)$, contradicting the irreducibility of $R_w$ (and also  the fact that we have ruled out the possibility that  either projection of $R_w$ maps surjectively to $V$).

 So we may assume that ${\rm pr}_i(R_w)=C_i\subsetneq V$ is a curve, $i=1,2$.   But then we obtain an isomorphism $(\operatorname{pr}_1\times \operatorname{pr}_2):R_w\to C_1\times C_2\subseteq V\times V$, and composing with $g$ we would arrive at the conclusion that $W$ has a curve summand, which contradicts our assumption.  Therefore, up to showing that $R_w$ is irreducible, we have succeeded in showing  that $\deg(g)=2$ is not possible.

\begin{lemma} \label{lem:R_w=irred'} 
$R_w$ is irreducible.
\end{lemma}

\begin{proof}
The basic idea is similar to the arguments we have been using.  If $R_w$ is reducible for general $w$, then the family $\mathscr R_W\to W$ of $R_w$ has geometrically reducible generic fiber (with two geometric  components).  Therefore, after a generically finite (degree $1$ or $2$) base change by say $W'\rightarrow W$, we may assume that the base change  $\mathscr R_{W'}$ of the family $\mathscr R_{W}$ has two components $\mathscr R_{W',j}\subseteq V\times V\times W'$, $j=1,2$.  By construction each component is birational to $W\times W'$.
We may thus consider for $i=1,2$ the induced rational maps 
$$
W\times W' \dashrightarrow \mathscr R_{W',j}\stackrel{\operatorname{pr}_i}{\longrightarrow}V ,
$$ 
then using the universal property of the  Albanese and the property $(\mathcal P)$ one obtains that  $V$ has a curve summand, giving a contradiction. 

The details are as follows.  Let $w\in W$ be general, and 
let us assume that 
\[
R_w=R_{w,1}\cup R_{w,2}
\] 
is reducible.  
Since $g|_{R_w}:R_w\rightarrow W-w$ has degree two, $R_{w,j}$ is birational via $g$ to $W-w$ for $j=1,2$.
Since $W$ has  property $(\mathcal P)$ with respect to $V$   by \theoremref{thm:propP}, \lemmaref{lem:debarre:propP:W-W} implies that for some $i_j\in\{1,2\}$,
$$
{\rm pr}_{i_j}:R_{w,j}\longrightarrow V
$$
is not surjective. 
By our assumptions $W$ is not a translate of $\pm V$ (which would be the case if the projection were a point), and so the above projection has image a curve 
$C_{w,i_j,j}\subsetneq V$.   

As usual, we are interested in the maps 
\[
\varphi_{w,i,j}:W\stackrel{\sim}{\dashrightarrow} R_{w,j}\stackrel{{\rm pr}_{i}}\longrightarrow V
\]
as $w$ varies in $W$.    The above discussion implies that for general $w\in W$, and $j=1,2$,  then for  some $i_j\in \{1,2\}$ we have $\operatorname{im}(\varphi_{w,i_j,j})=C_{w,i_j,j}$ is  a curve.

Now, as in \S \ref{S:GDPsi'}, we can construct
$$
\mathscr R_W\subseteq V\times V\times W
$$
with general fiber of the last  projection given by $R_w\times w$, as well as  a map 
$$
F:\mathscr R_W \longrightarrow A\times W
$$ 
over $W$ given by $F(v_1,v_2,w)=(w-(v_1-v_2),w)$, which has image in $W\times W$, and is generically $2:1$.  Since the general fiber $R_w$ is reducible, the  geometric generic  fiber of $\mathscr R_W$ is reducible, and there is a smooth projective surface $W'$ and a generically finite
 map $W'\rightarrow  W$ so that the base change $\mathscr R_{W'}=\mathscr R_{W',1}\cup \mathscr R_{W',2}$ is reducible (in fact, using the monodromy of the fibers, one can see that this can be taken to  be degree $1$ or degree $2$).   Via degree considerations on the fibers, the morphisms 
$$
F_j:\mathscr R_{W',j}\longrightarrow  W\times W'
$$
obtained from $F$ via base change, 
are birational.  

Now, for $i,j=1,2$, define the map $\Phi_{i,j}$ as the composition 
$$
\Phi_{i,j}:W\times W'\stackrel{F_j^{-1}}{\dashrightarrow} \mathscr R_{W',j} \stackrel{\operatorname{pr}_i}{\longrightarrow} V
$$
where $\operatorname{pr}_i$ is the projection onto the $i$-th factor.  
As in \S \ref{subsec:deg=1'}, we may use the Albanese to conclude that 
$$
\Phi_{i,j}=\varphi_{i,j}+c_{i,j}:W\times W'\longrightarrow V
$$
with $\varphi_{i,j}:W\to A$ and $c_{i,j}:W'\to A$.  
In particular, 
\begin{align} \label{eq:im(d_i)''}
V=\im (\varphi_{i,j})+\im (c_{i,j})
\end{align}
for $i,j=1,2$.  By construction, for general $w\in W$  we have $\varphi_{w,i,j}(\tilde w)=\varphi_{i,j}(\tilde w)+c_{i,j}( w')$ for all $\tilde w\in W$, where $w' \in W'$ maps to $ w$ under $W'\rightarrow  W$.  
Since we have shown  that for $j=1,2$ there is some $i_j\in\{1,2\}$ so that    $\operatorname{im}(\varphi_{w,i_j,j})=C_{w,i_j,j}$ is  a curve, \eqref{eq:im(d_i)''} establishes that $V$ has a curve summand, providing  the desired  contradiction.
\end{proof}

\subsection{Ruling out $\deg(g)=3$} \label{subsec:deg=3}
In this section we will derive a contradiction from the assumption that $\deg(g)=3$. 
By the same argument as at the beginning of \S\ref{subsec:deg=2'}, $\deg(g)=3$ implies  $(A,\Theta)$ is not isomorphic to the  Jacobian of a smooth curve, and  that $V$ does not admit  a curve summand.

We can further assume that $V$ is not a translate of $-V$.
Indeed, if $V$ were a translate of $-V$, then up to translation we may assume $V=-V$.
But then the difference map $g$ would factor through the involution 
$$V\times V\longrightarrow V\times V, \ \ (v_1,v_2)\mapsto (-v_2,-v_1), $$
 and so $g$ could not have degree $3$.

In order to present the main argument in the case $\deg(g)=3$, we introduce some notation.
For general $(v_1,v_2)\in V\times V$, the preimage $g^{-1}(v_1-v_2)$ consists of three points, which we denote by
\[
g^{-1}(v_1-v_2)=\left\{(v_1,v_2),(v_3,v_4),(v_5,v_6)\right\} .
\]
That is,
\[
v_1-v_2=v_3-v_4=v_5-v_6 .
\]
Also, for a general point $v_1\in V$ we will use the notation
\[
g^{-1}(v_1-V)^{\red}=(v_1\times V) \cup R_{v_1}\cup Q
\]
with $g(Q)\subsetneq v_1-V$, and where each component of $R_{v_1}$ dominates $v_1-V$ via $g$.

\begin{lemma} \label{lem:C(v_1)}
In the above notation, 
\begin{enumerate}
\item  If we fix a general point $v_1\in V$ and move $v_2\in V$, then the points $v_4$ and $v_6$ stay in an irreducible curve $C(v_1)\subseteq V$, whereas $v_3$ and $v_5$ sweep out $V$; i.e., $\operatorname{pr}_2(R_{v_1})=C(v_1)$ and $\operatorname{pr}_1(R_{v_1})=V$.  

\item For general $v_1$, we have that $R_{v_1}$ is irreducible.

\item There is a dense open susbset $U\subseteq V$ such that $\displaystyle \overline{\bigcup _{v_1\in U}C(v_1)}=V$.   
\end{enumerate}
\end{lemma}

\begin{proof}
We fix a general point $v_1\in V$ and use the notation introduced above.
Since for general $v_2\in V$, the preimage $g^{-1}(v_1-v_2)$ is given by three distinct points, the restriction
$$
g|_{R_{v_1}}:R_{v_1}\longrightarrow v_1-V
$$ 
has degree two. 
In particular, $R_{v_1}$ is nonempty.

Let now $R_{v_1}'\subseteq R_{v_1}$ be an irreducible component.
Since $V$ is not a translate of $-V$, ${\rm pr}_2(R_{v_1}')$ is not a point.
On the other hand, suppose that ${\rm pr}_2(R_{v_1}')$ is equal to $V$. 
Then for general $v_2\in V$ with $v_2\neq v_1$, there are $v_3,v_4\in V$ with $v_1-v_2=v_3-v_4$ and $v_1\neq v_3$, such that $v_4$ is a general point of $V$.
Hence, 
\[
v_1+v_4=v_2+v_3 ,
\]
with $[v_1,v_4]\neq [v_2,v_3]$, and where $v_1+v_4$ is a general point of $V+V$.    
Since $V\times V\longrightarrow V+V$ factors through $V^{(2)}$, it has by \lemmaref{lem:deg(f)} either degree $2$ or $6$ and so the above observation shows that it has degree $6$. 
Thus $V+V$ is a translate of $\Theta$, which contradicts $\deg(g)=3$ by  \theoremref{thm:V=W}. 

We have thus shown that ${\rm pr}_2(R_{v_1}')$ is a curve for each irreducible component $R_{v_1}'$ of $R_{v_1}$.
In the next step, we aim to prove that ${\rm pr}_2(R_{v_1})$ is irreducible.
For this it suffices to see that $R_{v_1}$ is irreducible.
If it is not, then as $g|_{R_{v_1}}:{R_{v_1}}\longrightarrow v_1-V$ has degree two, each component of ${R_{v_1}}$ is birational to $V$.
Since we know that for each component $R_{v_1}'$ of ${R_{v_1}}$, the projection ${\rm pr}_2(R_{v_1}')$ is an irreducible curve, the same argument as in the proof of \lemmaref{lem:R_w=irred'} then shows that $V$ has a curve summand, which contradicts our assumptions. 

It remains to see that $v_3$ (resp.\ $v_5$) sweeps out $V$ as $v_2$ moves.
If this were not true,  $\operatorname{pr}_i(R_{v_1})$ would be a curve not only for $i=2$ but also for $i=1$. 
Since we already know that ${R_{v_1}}$ is irreducible, it would follow that the natural inclusion  ${R_{v_1}}\subseteq \operatorname{pr}_1({R_{v_1}})\times \operatorname{pr}_2({R_{v_1}})$ is an equality so that $R_{v_1}$  is a product of curves.
Applying $g$ then shows that $V$ has a curve summand, which contradicts our assumptions.
This proves (1) of the lemma.

We have already shown (2), and (3) follows directly from the definition of $C(v_1)$, as for general $v\in V$, one can find $v_1$ with $v\in C(v_1)$.  
\end{proof}

For any general point $v_1\in V$, there is by \lemmaref{lem:C(v_1)} a rational map 
$$\phi_{v_1}: V\dashrightarrow  C(v_1)^{(2)}, \,\,\,\,\,\,\,\, v_2\mapsto [v_4,v_6].$$

\begin{lemma} \label{lem:phi_v1}
For general $v_1\in V$, the rational map $\phi_{v_1}$ is dominant. 
\end{lemma}

\begin{proof} 
We denote by $\im(\phi_{v_1})$ the (closure of the) image of $\phi_{v_1}$.
That is, 
$$
\im(\phi_{v_1})=\overline{\phi_{v_1}(U)}\subseteq C(v_1)^{(2)}
$$ 
for some Zariski open and dense subset $U\subset V$ on which $\phi_{v_1}$ is defined.

Let us first prove that $\im(\phi_{v_1})$ is not a point.
That is, we want to prove that for $v_2\in V$ general, $\phi_{v_1}(v_2)=[v_4,v_6]$ is not constant as $v_2$ moves.
If it were, then $v_4$ and $v_6$ would not move, although $v_2$ does.
This would imply that ${R_{v_1}}$ is reducible (and also that $V$ is a translate of $-V$), which contradicts our assumptions. 

It remains to prove that $\im(\phi_{v_1})$ is not a curve.
For a contradiction, we assume that $\im(\phi_{v_1})$ is a curve for general $v_1\in V$.

For a general point  $(v_1,v_2)\in V\times V$, let us again look at
\begin{align} \label{eq:v1-v2}
g^{-1}(v_1-v_2)=\left\{(v_1,v_2),(v_3,v_4),(v_5,v_6)\right\} .
\end{align}
Since $\im(\phi_{v_1})\subset C(v_1)^{(2)}$ is an irreducible curve, it follows that the intersection
$$
[v_4\times C(v_1)]\cap \im(\phi_{v_1})
$$
is given by a finite set of points.
Indeed, if this was not the case, then $ \im(\phi_{v_1})= [v_4\times C(v_1)]$ which implies that if we fix $v_1$ and move $v_2$, then $v_4$ does not move and so $V$ is a translate of $-V$, which contradicts our assumptions.
We have thus proven that the above intersection is finite, which means that $v_6$ is determined by $v_1$ and $v_4$ up to finite ambiguity.

On the other hand, $v_1-v_2=v_3-v_4$ is a general point of $V-V$.
Therefore, $v_4-v_3\in V-V$ is a general point and so we can compute its preimage via $g$ as 
\begin{align} \label{eq:v4-v3}
g^{-1}(v_4-v_3)=\left\{(v_4,v_3),(v_2,v_1),(v_6,v_5)\right\} .
\end{align}
Since $v_4-v_3$ is a general point of $V-V$, it follows that $v_4$ is a general point of $V$.\footnote{The conclusion that for general $(v_1,v_2)\in V\times V$, the point $v_4$ is a general point of $V$ is equivalent to saying that $v_4$ sweeps out $V$ as we move $(v_1,v_2)$ in $V\times V$.
This should not be confused with \lemmaref{lem:C(v_1)}, where we have seen that $v_4$ stays in a curve if we fix $v_1$ and move $v_2$.} 
Let us now fix the general point $v_4$ and move $v_3\in V$.
It then follows from \lemmaref{lem:C(v_1)} that $v_1,v_5$ stay in the irreducible curve $C(v_4)$, whereas $v_2$ and $v_6$ sweep out $V$.

Note that locally analytically, the labeling of points in \eqref{eq:v1-v2} and \eqref{eq:v4-v3} is consistent, and the discussion above shows that fixing $v_4$, then for small changes in $v_3$, the point $v_1$ moves in a one-dimensional family along the curve $C(v_4)$ while $v_6$ moves in a two-dimensional family.
This is a contradiction, because we have seen earlier that for fixed $v_4$, the point  $v_6$ is determined up to finite ambiguity by $v_1$ and so, for fixed $v_4$, the locus of all possible $v_6$'s is one-dimensional, since $v_1\in C(v_4)$ can only move along a curve.
This  finishes the proof of the lemma.
\end{proof}

Let $\widetilde{C}(v_1)$ denote the normalization of $C(v_1)$.
By \lemmaref{lem:phi_v1}, $\phi_{v_1}$ induces a dominant rational map 
$$ 
V\dashrightarrow \widetilde{C}(v_1)^{(2)}.
$$
Since $H^1$ is a birational invariant of smooth varieties, this rational map induces an injective morphism of integral Hodge structures
\[
H^1(\widetilde{C}(v_1),\ZZ)=H^1(\widetilde{C}(v_1)^{(2)},\ZZ)\longrightarrow H^1(V,\ZZ) ,
\]
which induces an inclusion of the family of Jacobians $J(\widetilde C(v_1))$ in the fixed abelian variety $\operatorname{Pic}^0(V)$. 
This implies that the abelian varieties $J(\widetilde C(v_1))$ are all isomorphic, as $v_1$ moves; since polarizations of a fixed abelian variety are discrete, we can then use the Torelli theorem to conclude that the family of curves $(\widetilde{C}(v_1))_{v_1\in U}$, where $v_1$ runs through the points of some Zariski open and dense subset $U\subseteq V$, is isotrivial.  Let us denote this isotrivial  family of curves $\mathscr C\to U$, and fix $\widetilde C\cong \widetilde C(v_1)$ for some (any) $v_1\in U$.

There exists a generically finite morphism $V'\rightarrow V$ such that the pullback of  $\mathscr C\to U$  to some dense open subset $U'\subseteq V'$ is trivial.  In other words, 
$\mathscr C\times_U U'\cong U'\times \widetilde C$.  We restrict the  trivial family  $V'\times C'$ to a general curve $B\subseteq V'$ to obtain the family $B\times \widetilde C\to B$, and then consider the rational map $B\times \widetilde C \dashrightarrow V$ which is fiberwise over $v_1\in B\cap U'$  induced by 
$$
\widetilde C\cong \widetilde C(v_1)\longrightarrow C(v_1)\subseteq V.
$$
As $B$ is general, $\bigcup _{v_1\in B\cap U'} C(v_1)$ is dense in $V$ by \lemmaref{lem:C(v_1)}.  
Consequently,  $B\times \widetilde C \dashrightarrow V$  is  a dominant rational map from a product of curves to $V$, and so $V$ has a curve summand; see for instance \cite[Corollary 22]{Schreieder1}.
This contradicts our assumptions, which finishes the proof of \theoremref{thm:degg}, and hence also the proof of \theoremref{main}.

\begin{remark}
The smoothness assumption on $V$ and $W$ in \theoremref{thm:degg} is not necessary.
This follows from \theoremref{thm:V=sing}, but also from the proof of \theoremref{thm:degg} given above, which can easily be adapted to the case where $V$ and $W$ are allowed to be singular.
\end{remark}


\section*{References}
\begin{biblist}

\bib{Beauville1}{article}{
      author={Beauville, Arna{-}ud},
      title={Prym varieties and the Schottky problem},
	journal={Invent. Math.}, 
	number={41},
	date={1977}, 
	pages={149--196},
}
\bib{Beauville2}{book}{
        author={Beauville, Arnaud},
        title={Surfaces alg\'ebriques complexes},
        series={Ast\'erisque}, 
        number={54}, 
        publisher={Soc. Math. de France}, 
        date={1978},
}
\bib{Beauville3}{article}{
      author={Beauville, Arna{-}ud},
      title={L'In\'egalit\'e $p_g \ge 2q - 4$ pour les surfaces de type g\'en\'eral, \emph{Appendix to ``In\'egalit\'es num\'eriques
      pour les surfaces de type g\'en\'eral" by O. Debarre}},
	journal={Bull.  Soc. Math. Fr.}, 
	number={110},
	date={1982}, 
	pages={343--346},
}
\bib{BL}{book}{
        author={Birkenhake, Christina},
        author={Lange, Herbert},
        title={Complex abelian varieties, 2nd ed.},
        publisher={Springer, Heidelberg}, 
        date={2004},
}
\bib{Casalaina1}{article}{
        author={Casalaina-Martin, Seba{-}stian},
        title={Cubic threefolds and abelian varieties of dimension five II},
        journal={Math. Z.}, 
        number={260}, 
        date={2008},
        pages={115--125},
}
\bib{Casalaina2}{article}{
        author={Casalaina-Martin, Sebastian},
        title={Singularities of the Prym theta divisor},
        journal={Ann. Math.}, 
        number={170}, 
        date={2009},
        pages={162--204},
}
\bib{CG}{article}{
      author={Clemens, Herb},
      author={Griffiths, Phillip},
	title={The intermediate Jacobian of the cubic threefold},
	journal={Ann. Math.}, 
	number={95},
	date={1972}, 
	pages={281--356},
}
\bib{Debarre}{article}{
        author={Debarre, Olivier},
        title={Minimal cohomology classes and Jacobians},
        journal={J. Algebraic Geom.}, 
        number={4}, 
        date={1995},
        pages={321--335},
}
\bib{Debarre2}{book}{
        author={Debarre, Oliv{-}ier},
        title={Complex tori and abelian varieties},
        series={SMF/AMS Texts and Monographs}, 
        number={11},
        publisher={Amer. Math. Soc., Providence RI}, 
        date={2005},
}
\bib{EL}{article}{
      author={Ein, Lawrence},
      author={Lazarsfeld, Robert},
	title={Singularities of theta divisors and the birational geometry of irregular varieties},
	journal={J. Amer. Math. Soc.}, 
	number={10},
	date={1997}, 
	pages={243--258},
}
\bib{Fulton}{book}{
        author={Fulton, William},
        title={Intersection theory, 2nd ed.},
        series={Ergebnisse der Mathematik und  ihrer Grenzgebiete}, 
        publisher={Springer-Verlag, Berlin}, 
        date={1998},
}
\bib{harris}{book}{
    AUTHOR = {Harris, Joe},
     TITLE = {Algebraic geometry},
    SERIES = {Graduate Texts in Mathematics},
    VOLUME = {133},
      NOTE = {A first course,
              Corrected reprint of the 1992 original},
 PUBLISHER = {Springer-Verlag, New York},
      YEAR = {1995},
     PAGES = {xx+328},
      ISBN = {0-387-97716-3},
   MRCLASS = {14-01},
  MRNUMBER = {1416564 (97e:14001)},
}
\bib{Hoering1}{article}{
      author={H\"oring, Andreas},
	title={$M$-regularity of the Fano surface},
	journal={C. R. Math. Acad. Sci. Paris}, 
	number={344},
	date={2007}, 
	pages={691--696},
}
\bib{Hoering2}{article}{
      author={H\"oring, Andr{-}eas},
	title={Minimal classes on the intermediate Jacobian of a generic cubic threefold},
	journal={Comm. Contemp. Math.}, 
	number={12},
	date={2010}, 
	pages={55-70},
}
\bib{LT}{article}{
      author={Lombardi, Luigi},
      author={Tirabassi, Sofia},
	title={$GV$-subschemes and their embeddings in principally polarized abelian varieties},
	journal={Math. Nachrichten}, 
	number={288},
	date={2015}, 
	pages={1405--1412},
}
\bib{Mumford}{article}{
      author={Mumford, David},
	title={Prym varieties I},
	journal={in \emph{Contributions to Analysis (a collection of papers
dedicated to Lipman Bers)}, New York, Academic Press}, 
	date={1974}, 
	pages={325-350},
}
\bib{PP1}{article}{
      author={Pareschi, Giuseppe},
      author={Popa, Mih{-}nea},
	title={Regularity on abelian varieties I},
	journal={J. Amer. Math. Soc.}, 
	number={16},
	date={2003}, 
	pages={285--302},
}
\bib{PP2}{article}{
      author={Pareschi, Giuseppe},
      author={Popa, Mihnea},
	title={Castelnuovo theory and the geometric Schottky problem},
	journal={J. Reine Angew. Math.}, 
	number={615},
	date={2008}, 
	pages={25--44},
}
\bib{PP3}{article}{
      author={Pareschi, Giuseppe},
      author={Popa, Mih{-}nea},
	title={Generic vanishing and minimal cohomology classes on abelian varieties},
	journal={Math. Ann.}, 
	number={340},
	date={2008}, 
	pages={209--222},
}
\bib{PP4}{article}{
      author={Pareschi, Giuseppe},
      author={Popa, Mihnea},
	title={$GV$-sheaves, Fourier-Mukai transform, and generic vanishing},
	journal={Amer. J. Math.}, 
	number={133},
	date={2011}, 
	pages={235--271},
}
\bib{Ran}{article}{
      author={Ran, Ziv},
      	title={On subvarieties of abelian varieties},
	journal={Invent. Math.}, 
	number={62},
	date={1981}, 
	pages={459--479},
}
\bib{Ran2}{article}{
      author={Ran, Zi{-}v},
      	title={A characterization of five-dimensional Jacobian varieties},
	journal={Invent. Math.}, 
	number={67},
	date={1982}, 
	pages={395--422},
} 
\bib{Schreieder1}{article}{
      author={Schreieder, Stefan},
      	title={Theta divisors with curve summands and the Schottky problem},
	journal={preprint: arXiv:1409.3134 (to appear in Math.\ Ann.)}, 
	date={2014}, 
}
\bib{Schreieder2}{article}{
      author={Schreieder, Stef{-}an},
      	title={Decomposable theta divisors and generic vanishing},
	journal={preprint: arXiv:1602.06226 (to appear in IMRN)}, 
	date={2016}, 
}






\end{biblist}
\end{document}